\documentclass[12pt,leqno]{article}

\usepackage{palatino}
\usepackage[mathcal]{euler}
\usepackage{multirow}
\usepackage{xcolor}

\usepackage{graphicx} 
\usepackage{epstopdf,epsfig}

\usepackage{comment} 
\usepackage{amsmath,amsthm,amsfonts,amssymb,latexsym,amscd,enumerate,url,hyperref}
\usepackage{amssymb}
\usepackage{url}

\usepackage[all]{xy}
\usepackage{multirow}

\theoremstyle{plain}
\newtheorem{theorem}{Theorem}[] 
\newtheorem{conjecture}{Conjecture}[] 
 \newtheorem{corollary}{Corollary}[subsection] 
\newtheorem{lemma}{Lemma}[subsection] 
\newtheorem{proposition}{Proposition}[subsection] 

\theoremstyle{definition}
\newtheorem{defn}{Definition}[subsection] 
\newtheorem{example}{Example}[subsection] 

\theoremstyle{remark}
\newtheorem{remark}{Remark}
\def\R{\mathbb R}
\def\C{\mathbb C}

\def\Z{\mathbb Z}

\def\P{\mathbb P}
\def\K{\boldsymbol{K}}

\def\g{\boldsymbol{\mathfrak{g}}}
\def\k{\boldsymbol{\mathfrak{k}}}

\def\h{\boldsymbol{\mathfrak{h}}}

\def\X{\mathbb{X}}
\def\fg{\mathfrak{g}}
\def\ft{\mathfrak{t}}

\def\fa{\mathfrak{a}}
\def\fk{\mathfrak{k}}
\def\fn{\mathfrak{n}}
\def\fp{\mathfrak{p}}

\def\fh{\mathfrak{h}}
\def\O{\mathcal{O}}
\def\Ox{\mathcal{O}_{\mathbb{X}}}
\def\F{\mathcal{F}}

\begin{document}

\title{The algebraic Mackey-Higson bijections}
 
\author{Eyal Subag\thanks{Department of Mathematics, Penn State University, University Park, PA 16802, USA.} }

 
 
 
\date{}

\maketitle


\begin{abstract}
\noindent 
For a connected semisimple  Lie group $G$ we describe  an explicit collection of  correspondences between the  admissible dual of $G$ and the  admissible dual of the Cartan motion group associated with $G$.  
We  conjecture that each of these correspondences induces  an algebraic isomorphism between the admissible duals. The constructed correspondences are defined in terms of algebraic families  of Harish-Chandra modules. We prove that the conjecture holds  in the case of $SL_2(\R)$, and in that case we give an equivalent  characterization for the bijections.
\end{abstract}

\setcounter{tocdepth}{2}
\tableofcontents


\newcommand\sfrac[2]{{#1/#2}}

\newcommand\cont{\operatorname{cont}}
\newcommand\diff{\operatorname{diff}}


\section{Introduction}
A fundamental question in representation theory of reductive Lie groups is the classification of  $\widehat{G}$, the unitary dual of  $G$. 
In the early 70's George Mackey described in some cases a partially defined one-to-one correspondence between $\widehat{G}$, and  $\widehat{G}_C$ where $G_C$ is the  Cartan motion group associated with $G$. In fact, around the same time  Quillen noticed that such a phenomenon exists for the Lorentz group \cite[p. 10]{Quillen}. Back to Mackey,  he  speculated that in general there should be a bijection between the unitary duals \cite{Mackey75}.   The  idea was to use the easy description of $\widehat{G}_C$,  via  ``the Mackey machine",   along with the 'bijection' to shed new light on $\widehat{G}$.   Mackey was motivated by the fact that $G_C$ can be obtained as a  limit of a family of groups that are isomorphic to $G$. Such a limit was used by the physicists and is known as \textit{contraction} \cite{ Segal51,Inonu-Wigner53,Gilmore05}. In some cases $G$, $G_C$, and their unitary irreducible representations all carry  a physical meaning that suggest some guidance in building the correspondence. About 40 years later  Higson constructed a decent bijection between $\widehat{G}^{\text{temp}}$, the tempered dual of $G$, and that of  $G_C$  in the case of a connected complex semisimple Lie group \cite{Higson08}. Moreover Higson showed how the existence of the bijection  leads to a new proof for  the Baum-Connes conjecture for complex semisimple Lie groups. A few years later he extended the bijection to an algebraic isomorphism  between the admissible duals, again, in the case of a connected complex semisimple Lie group \cite{Higson2011}. Recently, Afgoustidis made a major progress and constructed a bijection between $\widehat{G}^{\text{temp}}$ and $\widehat{G}_C^{\text{temp}}$ in the case of a real connected semisimple Lie group $G$ \cite{Afgoustidis15}.

The purpose of this paper is to study  Mackey-Higson-Afgoustidis type bijections between the admissible duals  $\widehat{G}^{\text{adms}}\longleftrightarrow \widehat{G}_C^{\text{adms}}$  in the case of a  real connected semisimple Lie group $G$. We shall construct a collection of  correspondences between admissible representations of $G$ and admissible representations of $G_C$ using certain algebraic families of Harish-Chandra modules (similar families were introduced in  \cite{Bernstein2016,Ber2017}). We  conjecture that these correspondences are indeed  algebraic isomorphisms between the admissible duals. We prove it in the case of  $SL_2(\R)$. Moreover,  still in the case of $SL_2(\R)$, we  are able to characterize the constructed  bijections.  This suggests  that  a similar characterization  hold in the general case and that our construction is a possible way to define what should be considered as  Mackey-Higson-Afgoustidis type bijections. This puts the known bijections obtained by Higson and Afgoustidis \cite{Higson08,Higson2011,Afgoustidis15} within a larger context where they have a better chance to be  characterized by their properties. Moreover, it suggests an approach for dealing with the non-uniqueness of the Mackey bijection. The bijections constructed  by Higson and Afgoustidis  can be slightly modify to give a family of bijections with similar properties. In our description all of these modifications are equally good.
   
A main theme in our approach is the use of algebraic families. It is interesting to note that when we lift classical notions as infinitesimal character and Harish-Chandra homomorphism to the 'family setup' the result is a new phenomenon in which representation theory of $G$ interacts with the geometry of the families, see Section \ref{se3.3} and Definition \ref{d4}. 
It is our opinion that the suggested perspective that all relevant objects sit on natural algebraic families together with the new interaction between   representation theory and the geometry of the families are novel ideas that should be studied further.

We shall now state  our results more carefully. Let $G$ be a complex semisimple Lie group  with an algebraic involution $\Theta$ and a fixed point subgroup $K$. Let  $\sigma$ be an anti-holomorphic involution of $G$ corresponding to $\Theta$. We assume that  the fixed point subgroup of $\sigma$, $G^{\sigma}$, is a real connected semisimple  group with finite center.  We further assume that $K^{\sigma}$ is a compact real form of $K$. We denote the corresponding complex Lie algebras by $\fg$ and $\fk$ and their real forms by   $ \fg^{\sigma}$ and $\fk^{\sigma}$. Using a variant of the deformation to the normal cone construction, we obtain an algebraic  family of Harish-Chandra pairs  $(\g,\K)$ over the base $\X:=\C\P^1$.  The fibers of $(\g,\K)$ satisfy $$(\g|_z,\K|_z)\simeq \begin{cases}
(\fg,K), & z\neq \infty \\
(\fk\ltimes \fp,K), & z=\infty
\end{cases} $$ where  $\fp$ is the minus one eigenspace of $\Theta$ in $\fg$ and $\fk\ltimes \fp$ is the Lie algebra of the (complex) Cartan motion group, $G_C=K\ltimes \fp$.  As a sheaf of $O_{\X}$-modules $$\g\simeq(\mathcal{O}_{\X}\otimes_{\C}\fk)\oplus  (\mathcal{O}_{\X}(-1)\otimes_{\C}\mathfrak{p}) $$ 
This means that $\g$ naturally decomposes into a constant family $\mathcal{O}_{\X}\otimes_{\C}\fk$ and a non-constant complement $\mathcal{O}_{\X}(-1)\otimes_{\C}\fp$. This is an instance for the interaction of the group structure with  the geometry of the family. 
For any $\Theta$-stable Cartan subalgebra $\fh$ of $\fg$ there is a corresponding Cartan subfamily of $\g$ denoted by $\h$.  Such a Cartan subfamily has a canonical decomposition into a constant and a non-constant subfamilies analogous to the decomposition of $\g$. 
We  define a \textit{generalized Harish-Chandra homomorphism} with respect to $\fh$, $\widetilde{\gamma}_{\fh}:\mathcal{Z}(\g)\longrightarrow \mathcal{U}(\h)$, as a (rational) morphism of $O_{\X}$-algebras from the center of the sheaf of enveloping algebras of $\g$ to the sheaf of universal enveloping algebras of $\h$, see Section \ref{se3.3}. Unlike the classical case, there are non-isomorphic Cartan subfamilies and the generalized Harish-Chandra homomorphism depends on the Cartan subfamily.

There is an obvious notion for  a family  of Harish-Chandra modules  for $(\g,\K)$. We explain what it means for such a family $\F$ to have an \textit{infinitesimal character  with respect to  $\fh$}, see Section \ref{sec4}. It is not the case that every generically irreducible  family of Harish-Chandra modules has an infinitesimal character. 
Another property of an algebraic  family  of Harish-Chandra modules is its  canonical isotypic decomposition with respect to the action of $K$:  $ \mathcal F = \bigoplus _{\mu\in \widehat K}   \mathcal{F}_\mu$. 
For any $\mu\in \widehat{K}$ Vogan attach a certain standard  tempered representation having a real infinitesimal character, $I_{\mu}$, and a $\Theta$-stable Cartan $\fh_{\mu}$, see Section \ref{Vog}. 
  We define $\widetilde{\mathcal{M}}(\g,\K)_{\mu}$ to be the collection of all generically irreducible $(\g,\K)$-modules  that are generated by their $\F_{\mu}$ part,
having an infinitesimal character with respect to $\fh_{\mu}$, and their fiber at $0$  equivalent in the  Grothendieck group of $(\fg,K)$ to  a submodule of  $I_{\mu}$. 
For $\F\in \widetilde{\mathcal{M}}(\g,\K)_{\mu}$ we  let  $J_{\mu,[\alpha:\beta]}(\mathcal{F})$ be the unique composition factor  of $\mathcal{F}|_{[\alpha:\beta]}$  that contains $\mu$. In Section \ref{Cor} we state the following conjecture.\newline

\noindent\textbf{Conjecture 1.}
Fix  $[\alpha:\beta]\in \R\P^1$ with $\alpha \beta \neq 0$. As $\mu$ varies in $\widehat{K}$ and  $\mathcal{F}$ varies in $\widetilde{\mathcal{M}}(\g,\K)_{\mu}$ the correspondence   $J_{\mu,[\alpha:\beta]}(\mathcal{F}) \longleftrightarrow J_{\mu,[0:1]}(\mathcal{F})$ defines a bijection between $\widehat{G^{\sigma}}^{\text{adms}}$ and $\widehat{K^{\sigma}\ltimes \fp^{\sigma}}^{\text{adms}}$ such that:
\begin{enumerate}
\item The bijection is an algebraic isomorphism, namely: For each $\mu\in \widehat{K}$ the bijection restricts to an isomorphism of affine algebraic varieties between $\widehat{G^{\sigma}}_{\mu}^{\text{adms}}$ and $ \widehat{K^{\sigma}\ltimes \fp^{\sigma}}_{\mu}^{\text{adms}}$.
\item The bijection extends Vogan's bijection (Theorem  \ref{th3}) between the tempered dual of $\widehat{G^{\sigma}}$  with real infinitesimal character    to $\widehat{K^{\sigma}}\subset \widehat{K^{\sigma}\ltimes \fp^{\sigma}}$.
\item The bijection maps tempered representations to tempered representations. 
\end{enumerate}

\noindent In addition we suggest a relatively easy way to calculate the various $J_{\mu,[\alpha:\beta]}(\mathcal{F})$ using the Jantzen filtration. The real structure $\sigma$ of $G$ induces a real structure on the family  $(\g,\K)$. For any family of Harish-Chandra modules $\F$ there is  a dual family $\F^{\sigma}$ that is defined with respect to   the real structure and called the $\sigma$-twisted dual , see Section \ref{s5.4}. A non-zero rational intertwining operator between $\F$ and $\F^{\sigma}$ gives rise to a canonical filtration on every $\F|_{[\alpha:\beta]}$. For $\F\in \widetilde{\mathcal{M}}(\g,\K)_{\mu}$ we denote by $\widetilde{J}_{\mu,[\alpha:\beta]}(\mathcal{F})$ the unique Jantzen quotient of $\F|_{[\alpha:\beta]}$ that contains   $\mu$.\newline

\noindent\textbf{Conjecture 2.}
 For any $[\alpha:\beta]\in \R\P^1$ and any $\F\in \widetilde{\mathcal{M}}(\g,\K)_{\mu}$  for which $\F|_{[\alpha:\beta]}$ is reducible,     \hspace{3mm}  $\widetilde{J}_{\mu,[\alpha:\beta]}(\mathcal{F})\simeq J_{\mu,[\alpha:\beta]}(\mathcal{F})$.\\

\noindent In Section \ref{sl2} we consider the case of $SL_2(\R)$ and  prove that\newline

\noindent\textbf{Theorem 5.}
Conjecture \ref{conj} and \ref{conj2} hold for $SL_2(\R)$.\\

\noindent In addition  for $SL_2(\R)$  we prove the following characterization of the constructed correspondences.

\noindent\textbf{Theorem 6.}
Any bijection  between $\widehat{SL_2(\R)}^{\text{adms}}$ and $\widehat{SO(2)\ltimes \R^2}^{\text{adms}}$ satisfying the  three  conditions in Conjecture 1  arises form the correspondence  $J_{\mu,[\alpha:\beta]}(\mathcal{F}) \longleftrightarrow J_{\mu,[0:1]}(\mathcal{F})$ for some $[\alpha:\beta]\in \R\P^1$ with $\alpha \beta \neq 0$.\\

\noindent  The Mackey bijection for the tempered duals in the case of $SL_2(\R)$ was studied before at \cite{George}.  Recently it was reformulated in terms of twisted $\mathcal{D}$-modules on the flag variety of $SL_2(\R)$ \cite{Qijun}. 
Most of the formalism of algebraic groups that we use here was developed in \cite{Bernstein2016,Ber2017}. We shall use Sections \ref{s2} and \ref{sec4} to recall some facts from there. In addition a lot of the calculations that we shall need were done, in some form or another, in  those references. In Section \ref{s3} we define the generalized Harish-Chandra homomorphisms. Section \ref{Cor} is used to state our conjectures and in Section  \ref{sl2} we deal with the case of $SL_2(\R)$.  \newline

The author is grateful to  Nigel Higson for introducing him to the subject of the \textit{Mackey-bijection} and for many useful conversations.  The author would also like to thank Joseph Bernstein for fruitful discussions.     

\section{The deformation  families}\label{s2}
In this section we recall and elaborate on  the construction of the deformation family of Harish-Chandra pairs as given in \cite[sec. 2.1.2]{Bernstein2016}.
That is, starting with a   complex Harish-Chandra pair $(\fg,K)$ we construct  a canonical family of Harish-Chandra  pairs, $(\g,\K)$, over the complex projective line, $\mathbb{X}=\C\P^1$, such that: 
\begin{enumerate}
\item Over $\C\subset \X$ the family is the constant family with fiber $(\fg,K)$.
\item At $\infty$ (the complement of  $\C$ in $\X$) the fiber is the Harish-Chandra pair consisting of the Lie algebra of the Cartan motion group $K\ltimes \fg/\fk$ and the group $K$.
\end{enumerate} 
When $\fg$ is the Lie algebra of a complex semismiple algebraic group $G$, $K$ the fixed point set of $\Theta$ (an algebraic involution of $G$), and $\sigma$ the corresponding real structure on $G$  we show that there is an induced  real structure  on the family of Harish-Chandra pairs. This real structure gives rise to a family of real groups parameterized  by $\X_{\R}:=\R\P^1$ with generic fiber equal to  $G^{\sigma}$  and the fiber at $\infty\in \X_{\R}$  given by the real Cartan motion group $K^{\sigma}\ltimes \fg^{\sigma}/\fk^{\sigma}$.

\subsection{The family of Harish-Chandra pairs}
Let $\fg$  be a complex semisimple Lie algebra with subalgebra $\fk$. Let $\g_{\text{const}}$ be the constant family of  Lie algebras over $\X=\C\P^1$ with fiber $\fg$. Explicitly,  $\g_{\text{const}}$ is the sheaf of (algebraic) sections of the bundle $\X \times \fg\longrightarrow \X$. Thinking of $\X$ as $\C\cup \{\infty\}$ we define  $\g$ to be the smallest  subsheaf of $\g_{\text{const}}$ containing all sections  that their values at infinity lie in $\fk$. Explicitly,  for any open $U\subset \X$ \begin{eqnarray}\nonumber
&& \Gamma(U,\g)=\begin{cases}
\Gamma(U,\g_{\text{const}}), & \infty \notin U\\
\{\xi\in \Gamma(U,\g_{\text{const}})| \xi(\infty)\in \fk\}, & \infty \in U
\end{cases}
\end{eqnarray}   
Picking any vector space complement, $\mathfrak{s}$, to $\fk$ in $\fg$ it follows that $$\g\simeq(\mathcal{O}_{\X}\otimes_{\C}\k)\oplus  (\mathcal{O}_{\X}(-1)\otimes_{\C}\mathfrak{s}) $$  
as sheaves of  $\Ox$-modules. From now on we   shall assume that $\fg$ is the Lie algebra of a complex semisimple algebraic group $G$. We further assume that we are given  an algebriac   involution  $\Theta$ of $G$ (we shall also refer to it as Cartan involution). Where the  fixed point group  of $\Theta$, denoted by $K$, having  $\fk$ as its Lie algebra. This implies that  $(\fg,K)$ is a Harish-Chandra pair. 
We have a canonical choice for a complement to $\fk$, namely the minus one eigenspace of the Cartan decomposition:
$$\mathfrak{g}=\mathfrak{k}\oplus \mathfrak{p}$$  
 The corresponding decomposition of $\g$:  
 $$\g\simeq(\mathcal{O}_{\X}\otimes_{\C}\k)\oplus  (\mathcal{O}_{\X}(-1)\otimes_{\C}\mathfrak{p}) $$  
is decomposition of $K$-equivariant sheaves of  $\Ox$-modules. The family $(\g,\K)$,  with $\K$ being the constant group scheme over $\X$ with fiber $K$,   is an algebraic family of Harish-Chandra pairs in the sense of \cite[sec. 2.3]{Bernstein2016} and its fibers satisfy $$(\g|_z,\K|_z)\simeq \begin{cases}
(\fg,K), & z\neq \infty \\
(\fk\ltimes \fp,K)\simeq(\fk\ltimes \fg/\fk,K) , & z=\infty
\end{cases} $$ 
where $z$ is 'the natural coordinate' on $\X$.    In terms of the standard coordinates $r=\frac{\beta}{\alpha}$ on $\X_0:=\{[\alpha:\beta]\in \X|\alpha\neq 0\}$, and $R=\frac{\alpha}{\beta}$ on $\X_{\infty}:=\{[\alpha:\beta]\in \X|\beta\neq 0\}$ we have
 \begin{eqnarray}\nonumber
 &&\Gamma(\X_0,\g)=\C[r]\otimes_{\C}\fg\\ \nonumber
 &&\Gamma(\X_{\infty},\g)=(\C[R]\otimes_{\C}\fk)\oplus (R\C[R]\otimes_{\C}\fp)
\end{eqnarray}

\subsection{The real structure}  
Recall that for a complex  reductive  group $G$ there is a bijection between conjugacy classes of algebraic involutions and conjugacy classes of  real forms e.g., see \cite[Theorems 3\&4, pp.230--231]{OnishchikVinberg} and \cite{Adams2016}.  Let $\sigma$ be a real form of $G$ that corresponds to $\Theta$. We shall assume that $K^{\sigma}$ is a maximal compact subgroup of $G^{\sigma}$.    By abuse of notation we denote by  $\sigma$ the corresponding real form on $\fg$ .  The morphisms
\begin{eqnarray}\nonumber
&&\sigma_{\X}:\X\longrightarrow \overline{\X}, \hspace{5mm} [\alpha:\beta]\longmapsto [\overline{\alpha}:\overline{\beta}] \\\nonumber
&& \sigma_{\g}:\g\longrightarrow \sigma_{\X}^*\overline{\g}, \hspace{5mm} \xi\longmapsto \sigma(\xi \circ \sigma_{\X}) \\ \nonumber
&&  \sigma_{\K}:\K\longrightarrow \sigma_{\K}^*\overline{\K}, \hspace{5mm} \tau\longmapsto \sigma(\tau \circ \sigma_{\X})
\end{eqnarray} 
Determines a real form of the family $(\g,\K)$ in the sense of   \cite[sec. 2.5]{Bernstein2016}. 
The fixed point set of  $\sigma_{\X}$ is $\X_{\R}=\R\P^1$.
And we obtain  a family of real Harish-Chandra pairs over $\X_{\R}$ denoted by $(\g^{\sigma},\K^{\sigma})$ with fibers given by 
\begin{eqnarray}\label{eq4}
&&(\g^{\sigma}|_x,\K^{\sigma}|_x)\simeq \begin{cases}
(\fg^{\sigma},K^{\sigma}), & x\neq \infty \\
(\fk^{\sigma}\ltimes \fp^{\sigma},K^{\sigma}), & x=\infty
\end{cases} \end{eqnarray}   
Here $\fp^{\sigma}=\fp\cap \fg^{\sigma}$.
By a real family we  mean that we have a collection of real Harish-Chandra pairs parameterized by the topological space $\X_{\R}$ (as a subset of $\X$ equipped with its analytic topology). 
Obviously the family $(\g^{\sigma},\K^{\sigma})$ carries more structure. For example the family $\g^{\sigma}$ is a real vector bundle and the family   $\K^{\sigma}$ is a trivial bundle of Lie groups. We shall  make no  use of this extra structure .

\subsubsection{The family of real groups}
Using the deformation to the normal cone construction \cite[sec 2.6]{Fulton84} one can construct a corresponding family of complex algebraic groups over $\X$. Then using the real structure obtaining a family of Lie group over  $\X_{\R}$. The focus in this paper is on the admissible dual of a real semisimple Lie group and the admissible dual of its Cartan motion group. These duals are defined in terms of the corresponding complex Harish-Chandra pair and in this sense the group play no role.   
 We note that each of the real Harish-Chandra pairs in (\ref{eq4}) determines a unique Lie group up to an isomorphism. By this we obtain a Lie group for each $x\in \R\P^1$.  

\section{The generalized Harish-Chandra homomorphisms}\label{s3}
In this section we  show how the classical Harish-Chandra homomorphism of $\fg$ can be lifted to $\g$. In order to do that we shall explain how certain decompositions in $\fg$ can be lifted to $\g$. In contrast to the classical case, we show that isomorphic Cartan subalgebras of $\fg$  lifted to non-isomorphic  'Cartan subfamilies of $\g$'. In particular, the lifted Harish-Chandra homomorphism arising from a Cartan subalgebra $\fh\subset \fg$ is very much dependent on $\fh$.

\subsection{Subfamilies generated by a set}\label{se3.1}
Any vector $\xi$ of $\fg$ gives rise to a section $1\otimes \xi$ of  $\Gamma(\X_0,\g)$.  For a subset $A$ of  $\mathfrak{g}$  we denote by $S_A$ the subsheaf of $\g$ that is generated  by the  sections $\{ 1\otimes \xi| \xi \in A\}\subset \Gamma(\X_0,\g)$. In other words, $S_A$ is the smallest subsheaf of $\g$  (as a quasi-coherent sheaf of $O_{\X}$-modules) that contains the local sections $\{ 1\otimes \xi| \xi \in A\}$.

\begin{example}
Lat $A=\{\xi\}$ for some $\xi\in \fg$. If $\xi\in \fk$ then $S_{A}$ consists of all sections of $\g_{\text{const}}$ with values in $\C \xi$. If  $\xi\notin \fk$ then $S_{A}$ consists of all sections of $\g_{\text{const}}$ with values in $\C \xi$ that  vanish at  $\infty \in \X$. In particular,   for non zero $\xi$
$$S_{A}\simeq \begin{cases}
 O_{\X},& \xi\in \fk\\
 O_{\X}(-1), &  \xi\notin \fk \end{cases}$$   
\end{example}
\noindent More generally, we have the following  lemma.
\begin{lemma}
Let $V$ be a vector subspace of $\fg$. Then $S_{V}$ consists of all sections of $\g$ with values in
$V$ and such that their value at $\infty$ lie inside $V\cap \fk$. In particular if $l= \operatorname{dim}(V\cap \fk)$ and $m= \operatorname{dim}(V)-l$ then  $$S_{V}\simeq ( \underbrace{O_{\X}\oplus...   \oplus O_{\X}}_{l-\text{times}})\oplus (\underbrace{O_{\X}(-1)\oplus...   \oplus O_{\X}(-1)}_{m-\text{times}}) $$ \qed \label{lem1}
\end{lemma}
\subsection{Families of Cartan subalgebras and root space decomposition for $\g$ }\label{sec32}
In this section we show how to lift a root space decomposition of $\fg$ to $\g$, or more precisely, to most of $\g$.

From now on we assume that $\fh$ is a  $\Theta$-stable Cartan subalgebra   of a complex semisimple Lie algebra $\fg$, with  $\mathfrak{t}=\mathfrak{h}\cap \mathfrak{k}$  and  $\mathfrak{a}=\mathfrak{h}\cap \mathfrak{p}$. Let  $\mathfrak{g}=\mathfrak{h}\oplus_{\alpha\in \Delta(\mathfrak{h},\mathfrak{g})} \mathfrak{g}_{\alpha}$  be the corresponding root space decomposition. Choosing an order we get the corresponding  positive roots  $\Delta^+(\mathfrak{h},\mathfrak{g})$ and the corresponding  maximal nilpotent subalgebras $\mathfrak{n}=\oplus_{\alpha\in \Delta^+(\mathfrak{h},\mathfrak{g})}\mathfrak{g}_{\alpha}$ and  $\overline{\mathfrak{n}}=\oplus_{-\alpha\in \Delta^+(\mathfrak{h},\mathfrak{g})}\mathfrak{g}_{\alpha}$  and the decomposition $\fg=\overline{\mathfrak{n}}\oplus \fh \oplus \mathfrak{n}$.
The Cartan  subalgebra gives rise to an abelian subfamily of $\g$ namely $\h=S_{\mathfrak{h}}$. We shall call $\h$ the Cartan subfamily of $\g$ corresponding to $\fh$, or simply a Cartan family. As already mentioned in Lemma (\ref{lem1}) 
$$\h=S_{\mathfrak{h}}\simeq (O_{\X}\otimes_{\C}\mathfrak{t})\oplus (O_{\X}(-1)\otimes_{\C}\mathfrak{a})$$ 
and more explicitly,  the  sections over $\X_0$ and $\X_{\infty}$ are given by  $\Gamma(\X_0,\h)=\C[r]\otimes_{\C}\mathfrak{h}$, and  $\Gamma(\X_{\infty},\h)=(\C[R]\otimes_{\C}\mathfrak{t})\oplus (R\C[R]\otimes_{\C}\mathfrak{a})$.
Later on we shall be interested in the dual of $\h$ so it is useful to note now that  
$$\h^* :=\operatorname{Hom}_{O_{\X}}(\h,O_{\X})\simeq  (O_{\X}\otimes_{\C}\mathfrak{t}^*)\oplus (O_{\X}(1)\otimes_{\C}\mathfrak{a}^*)$$  
Similar to the classical case we obtain a corresponding 'root space decomposition' not of $\g$ but of a certain subfamily of it, $\g_{\fh}$. The family $\g_{\fh}$  is defined as  the subfamily of $\g$ consisting of all sections such that  their values at $\infty$ (if $\infty$ is in their domain)  lie inside  $(\mathfrak{k}\cap\mathfrak{h})\oplus_{\alpha\in \Delta(\mathfrak{h},\mathfrak{g})} (\mathfrak{k}\cap\mathfrak{g}_{\alpha})$:
\begin{eqnarray}
&&\g_{\fh}:= S_{\mathfrak{h}}\oplus_{\alpha\in \Delta(\mathfrak{h},\mathfrak{g})} S_{\mathfrak{g}_{\alpha}}
\label{23}
\end{eqnarray}
If $\fh\subset \fk$ then $\g=\g_{\fh}$ in general  $\g\supset \g_{\fh}$. 
We can promote each $\alpha \in \Delta(\mathfrak{h},\mathfrak{g})$ to an element $\widetilde{\alpha}\in \h^*$ by first extending $\alpha$ using extension of  scalars  to be a morphism from $O_{\X}\otimes \fh$ into $O_{\X}$ and then use restriction to $\h$.   Equation (\ref{23}) is indeed a root space decomposition in the following sense:  the family of abelian Lie algebras $\h=S_{\mathfrak{h}}$ is equal to its  centralizer  in $\g_{\fh}$, and 
$$\Gamma(U,S_{\mathfrak{g}_{\alpha}})=\{ \xi\in \Gamma(U,\g_{\fh})| [\eta,\xi]=\widetilde{\alpha}(\eta)\xi, \forall \eta \in \Gamma(U,\h)\}$$ for any open $U$ in $\X$.  

\subsection{The Harish-Chandra homomorphism}\label{se3.3}
Let  $\fh$ be  a $\Theta$-stable Cartan subalgebra of $\fg$  and let $\gamma_{\fh}:\mathcal{Z}(\fg)\longrightarrow \mathcal{U}(\fh)$ be the corresponding normalized\footnote{Including the rho shift with respect to $\fn$.} Harish-Chandra morphism of $\C$-algebras. We shall show that $\gamma_{\fh}$ can be lifted to a rational  morphism of  sheaves of commutative $O_{\X}$-algebras   $\widetilde{\gamma}_{\fh}:\mathcal{Z}(\g)\longrightarrow \mathcal{U}(\h)$ and in some cases prove that it is regular.

By extension of scalars $\gamma_{\fh}$ can be  lifted to a morphism  of  sheaves of commutative $O_{\X}$-algebras
 of the constant families: 
 \begin{eqnarray}\nonumber
 &\widetilde{\gamma}_{\fh}:&O_{\X}\otimes_{\C}\mathcal{Z}(\fg)\longrightarrow O_{\X}\otimes_{\C}\mathcal{U}(\fh)\\ \nonumber
 &&f\otimes \xi \longmapsto f\otimes \gamma_{\fh}(\xi)
 \end{eqnarray}
 The following lemma shows that $\mathcal{Z}(\g)$ is a subsheaf of the center of the constant family $\mathcal{Z}(\g_{\text{const}})=O_{\X}\otimes_{\C}\mathcal{Z}(\fg)$. 
 
\begin{lemma}
 $\mathcal{Z}(\g)= \mathcal{Z}(\g_{\text{const}})\cap \mathcal{U}(\g)$.\label{le}
\end{lemma}
\begin{proof}
Since $\Gamma(\X_0,\g)=\Gamma(\X_0,\g_{\text{const}})$ it is enough to show that  $\Gamma(\X_{\infty},\mathcal{Z}(\g))=\Gamma(\X_{\infty},\mathcal{Z}(\g_{\text{const}})\cap \mathcal{U}(\g))$. Let $s\in  \Gamma(\X_{\infty},\mathcal{Z}(\g))$ then  $s|_{\X_0\cap \X_{\infty}}\in  \Gamma(\X_0\cap \X_{\infty},\mathcal{Z}(\g))$ $=\Gamma(\X_0\cap \X_{\infty},\mathcal{Z}(\g_{\text{const}})\cap \mathcal{U}(\g))$. If $\tau \in  \Gamma(\X_{\infty}, \mathcal{U}(\g_{\text{const}}))$ then over $\X_0\cap \X_{\infty}$ we have $[s|_{\X_0\cap \X_{\infty}},\tau|_{\X_0\cap \X_{\infty}}]=0$ and from the continuity of $s$, $\tau$, and the brackets it follows that $[s,\tau]=0$. Hence $s\in  \Gamma(\X_{\infty}, \mathcal{Z}(\g_{\text{const}})\cap \mathcal{U}(\g))$.  The inclusion $\mathcal{Z}(\g)\supset (\mathcal{Z}(\g_{\text{const}})\cap \mathcal{U}(\g))$ is obvious.
\end{proof} 
 By restriction we obtain a morphism $\widetilde{\gamma}_{\fh}:\mathcal{Z}(\g)\longrightarrow O_{\X}\otimes_{\C}\mathcal{U}(\fh)$. This morphism is a morphism of sheaves of commutative $O_{\X}$-algebras. We can interpret it as a rational morphism $\widetilde{\gamma}_{\fh}:\mathcal{Z}(\g)\longrightarrow \mathcal{U}(\h)$. In section \ref{sl2} we will show that for any  $\Theta$-stable Cartan of  $SL_2(\R)$ the morphism $\widetilde{\gamma}_{\fh}:\mathcal{Z}(\g)\longrightarrow \mathcal{U}(\h)$ is a decent morphism of $O_{\X}$-modules. In the appendix we prove that for the most compact $\Theta$-stable Cartans $\widetilde{\gamma}_{\fh}$ is regular.  At the moment we can not prove that in general  $\widetilde{\gamma}_{\fh}$ is  regular but it seems very likely to be the case.

\begin{remark} 
It should be clear that in general one can not expect  $\widetilde{\gamma}_{\fh}|_{\mathcal{Z}(\g)}$ to be an isomorphism since in general, there might be two $\Theta$-stable Cartans $\fh_1$ and $\fh_2$ with $\h_1$ not isomorphic to $\h_2$.
\end{remark}

\section{Algebraic families of Harish-Chandra  modules}\label{sec4}
In this section we  set notations and recall some properties of algebraic families of Harish-Chandra  modules as they were defined in 
\cite{Bernstein2016}.
 Let $(\boldsymbol{\mathfrak{g}},{\boldsymbol{K}})$ be an algebraic family of   Harish-Chandra pairs over a quasi-projective complex algebraic variety $X$. An \emph{algebraic family of Harish-Chandra modules} for  $(\boldsymbol{\mathfrak{g}},{\boldsymbol{K}})$ is a flat, quasicoherent  $O_X$-module $\mathcal{F}$  that is equipped with an  action of $\boldsymbol{K}$, and an  action  of  $\boldsymbol{\mathfrak{g}}$,
such that  the action morphism
\[
 \boldsymbol{\mathfrak{g}}\otimes_{O_X}\mathcal{F}\longrightarrow \mathcal{F}
\]
 is $\boldsymbol{K}$-equivariant, and such that the differential of the $\boldsymbol{K}$-action on $\F$  is equal to the composition of the embedding of $\operatorname{Lie}(\boldsymbol{K})$ into $\boldsymbol{\mathfrak{g}}$   with the action of   $\boldsymbol{\mathfrak{g}}$ on $\mathcal{F}$. 
Assuming that the  family $\boldsymbol{K}$ is constant and reductive (which is the case considered in this paper),  $\mathcal{F}$  is said to be    \emph{quasi-admissible} if   
$[ \mathcal{L}\otimes_{O_X} \mathcal{F}]^{\boldsymbol{K}}$ is a locally free and finitely generated sheaf of $O_X$-modules  for every  family $\mathcal{L}$ of representations of $\boldsymbol{K}$ that is locally free and   finitely generated  as an $O_X$-module.
In this case, $\F$ is locally free and there is a canonical isotypic decomposition  
   \[
   \mathcal F = \bigoplus _{\tau\in \widehat K}   \mathcal{F}_\tau .
   \]
We shall denote the category of quai-admissible families of  Harish-Chandra modules for $(\g,\K)$ by $\mathcal{M}(\g,\K)$. When $X$ is a curve we say that $\mathcal{F}$ is \textit{generically irreducible} if up to at most a countable number of fibers all its fibers are irreducible.
We say that  $\mathcal{F}$ is  \textit{quasisimple} if the action of $\mathcal{Z}(\g)$, the center of the enveloping algebra of $\g$, on $\mathcal{F}$ factors through a morphism of sheaves of $O_X$-algebras $\chi_{\mathcal{F}}: \mathcal{Z}(\g)\longrightarrow O_X$. That is, the action of any local section $\zeta$ of $\mathcal{Z}(\g)$ is given by multiplication by the regular function $\chi_{\mathcal{F}}(\zeta)$. A  generically irreducible family of Harish-Chandra modules is quasisimple.
In the classical case, thanks to the Harish-Chandra isomorphism, characters of the center of the enveloping algebras are the same as linear functionals of $\fh$ (that are invariant under the action of the Weyl group). 	This means that any irreducible admissible Harish-Chandra module  has an infinitesimal character. In the context of families  this is not true. Below we define when does a family of Harish-Chandra modules has an infinitesimal character with respect to a given $\Theta$-stable Cartan.
We shall keep our assumptions on $\fg,\fk$ and $\Theta$ as before and denote by $(\boldsymbol{\mathfrak{g}},{\boldsymbol{K}})$  the corresponding deformation family of Harish-Chanadr pairs over $\X$.

\begin{defn}
Let $\fh$ be a $\Theta$-stable Cartan subalgebra of $\fg$ and   $\mathcal{F}$  a quasisimple family of Harish-Chandra modules for  $(\boldsymbol{\mathfrak{g}},{\boldsymbol{K}})$. The family $\mathcal{F}$ \textit{has an infinitesimal character (with respect to $\h$)} if there is a morphism of sheaves of $O_{\X}$-algebras  $\psi_{\mathcal{F}}:\mathcal{U}(\h)\longrightarrow O_{\X}$ such that  the diagram
$$\xymatrix{
 \mathcal{Z}(\g)\ar[d]_{ \widetilde{\gamma}_{\fh}}\ar[r]^{\hspace{1mm}\chi_{\mathcal{F}}} &O_{\X} \\
    \mathcal{U}(\h)\ar[ur]_{\psi_{\mathcal{F}}} &} $$
is commutative. \label{d4}
\end{defn}
In Section  \ref{sl2} we shall  see that a family $\mathcal{F}$ might have an infinitesimal character with respect to one $\Theta$-stable Cartan and not with respect  to another. In principal the space $\h^*=\operatorname{Hom}_{O_{\X}}(\h,O_{\X})$ is ``smaller" than $\operatorname{Hom}_{O_{\X}}( \mathcal{Z}(\g),O_{\X})$ so having an infinitesimal character is a non-trivial constrain on $\mathcal{F}$.

\section{The Correspondences between $\widehat{G^{\sigma}}$ and $\widehat{K^{\sigma}\ltimes \fp^{\sigma}}$}\label{Cor}
As before, and through out  this section we shall keep assuming the following set up. $G$ is a  semisimple  complex algebraic group with an algebraic involution $\Theta$ and a corresponding real from $\sigma$ such that:
\begin{enumerate}
\setlength\itemsep{0.01em}
\item $G^{\sigma}$ is a connected semisimple Lie group with finite center.
\item  $K^{\sigma}=G^{\sigma}\cap G^{\theta}$ is a maximal compact subgroup of $G^{\sigma}$.
\item  $K^{\sigma}$ is a compact real form of the reductive group $K:=G^{\theta}$.
\end{enumerate}
  The complex Lie algebras of $G$ and $K$ will be denoted by $\fg$ and $\fk$ respectively and their corresponding real forms by     $\fg^{\sigma}$ and $\fk^{\sigma}$ respectively. The associated Cartan decompositions of $\fg$ and $\fg^{\sigma}$ are 
\begin{eqnarray}\nonumber
&&\mathfrak{g}=\mathfrak{k}\oplus \mathfrak{p}\\ \nonumber
&&\mathfrak{g}^{\sigma}=\mathfrak{k}^{\sigma}\oplus \mathfrak{p}^{\sigma}
\end{eqnarray}
The  set of infinitesimal equivalence classes of irreducible admissible  representations of $G^{\sigma}$ is denoted by  $\widehat{G^{\sigma}}$.
In the first part of this section we recall the algebraic description of  $\widehat{G^{\sigma}}$ by minimal $K$-types due to Vogan. In the following subsection we define classes of algebraic families of Harish-Candra modules for $(\g,
\K)$. These classes are defined in terms of  the algebraic description of $\widehat{G^{\sigma}}$. Then we describe a correspondence  between $\widehat{G^{\sigma}}$ and $\widehat{K^{\sigma}\ltimes \fp^{\sigma}}$ and conjecture that the correspondence is an algebraic bijection.
On the last part of this section we rephrase the correspondence in  terms of   Jantzen filtrations.

\subsection{Vogan's classification}\label{Vog}

In this section we shall briefly  recall the algebraic description of $\widehat{G^{\sigma}}$ which is due to Vogan \cite{Vogan77,Vogan79,Vogan81}. 	Since $G^{\sigma}$ is connected so does $K^{\sigma}$  and the highest weight theory of Cartan and Weyl can be used to describe $\widehat{K^{\sigma}}$, e.g., see \cite[Chapter IV] {Knapp86}. This means that by choosing a maximal torus $T^{\sigma}$ of $K^{\sigma}$ any  $\mu\in \widehat{K^{\sigma}}$ is uniquely defined by  a certain   linear functional (highest weight) of $\ft$,  the comlexification of the Lie algebra of $T^{\sigma}$. Using the killing form of $\fg$ one can define a norm on $\ft^*$ which in turn defines a "norm", $\| \_\|:\widehat{K^{\sigma}}\longrightarrow [0,\infty)$. 
 A \textit{minimal  ${K^{\sigma}}$-type} of $\pi\in \widehat{G^{\sigma}}$ is  $\mu\in \widehat{K^{\sigma}}$ that  appears in  $\pi|_{\widehat{K^{\sigma}}}$ with nonzero multiplicity and minimizes the restriction of $\|\_ \|$ to those  ${K^{\sigma}}$-types that appear in  $\pi|_{\widehat{K^{\sigma}}}$, (see \cite[chapter X]{KnappVogan}).
 For every $\mu \in \widehat{K^{\sigma}}$ we denote by $\widehat{G^{\sigma}}_{\mu}$ the subset   of $\widehat{G^{\sigma}}$ consisting of those equivalence classes having $\mu$ as a minimal  ${K^{\sigma}}$-types. It is known that $G=\bigcup_{\mu \in \widehat{K^{\sigma}}}\widehat{G^{\sigma}}_{\mu}$. 

Let $P$ be a parabolic subgroup  of $G^{\sigma}$ with a Langlands decomposition $P=MAN$ (e.g., see \cite[chapter VII]{Knapp2002}). Denote the linear dual of the complexified Lie algebra of $A$ by $\fa^*$. For  any $\delta \in \widehat{M}$ and   $\nu \in \fa^*$ we denote the parabolically (normalized) induced representation $\operatorname{Ind}_P^{G^{\sigma}}\delta \otimes \nu \otimes 1 $ by $I_P(\delta,\nu)$.   The following  theorem gives  a complete description of $\widehat{G^{\sigma}}$.
\begin{theorem}[Vogan \cite{Vogan79} Theorem 1.1 and 1.2]
Let $\mu \in \widehat{K^{\sigma}}$. There exist  a  cuspidal parabolic subgroup of $G^{\sigma}$, $P_{\mu}=M_{\mu}A_{\mu}N_{\mu}$, and a discrete series representation $\delta_{\mu}$ of $M_{\mu}$ such that:
\begin{enumerate}
\item For any $\nu \in \fa^*$,   $I_{P_{\mu}}(\delta_{\mu},\nu)$ has a unique irreducible subquotient, $J_{P_{\mu}}(\delta_{\mu},\nu)$, that contains $\mu$ (as a minimal $K^{\sigma}$-type).
\item For any  $\pi\in \widehat{G^{\sigma}}_{\mu}$ there exists   $\nu \in \fa^*$  such that $\pi$ is infinitesimally equivalent to $J_{P_{\mu}}(\delta_{\mu},\nu)$.
\item  $J_{P_{\mu}}(\delta_{\mu},\nu)\simeq J_{P_{\mu'}}(\delta'_{\mu'},\nu')$ implies that  $(M_{\mu}A_{\mu}, \delta_{\mu},\otimes \nu)$ is conjugate to $(M_{\mu'}A_{\mu'}, \delta_{\mu'},\otimes \nu')$. Furthermore  assuming $P_{\mu}=P_{\mu'}$ and 
 $\delta_{\mu}=\delta_{\mu'}$ then  $J_{P_{\mu}}(\delta_{\mu},\nu)\simeq J_{P_{\mu}}(\delta_{\mu},\nu')$ if and only if $\nu'$ is obtained from $\nu$ under the action of a certain finite group (for details see \cite{Vogan79}).
\end{enumerate}
\label{th2}
\end{theorem}
\noindent Later on we shall  need  the following  related result of Vogan identifying a certain part of $\widehat{G^{\sigma}}$ with $\widehat{K^{\sigma}}$.   
\begin{theorem}[Vogan \cite{Vogan99} Theorem 8.1]For $G^{\sigma}$ a reductive Lie group the map that sends an irreducible tempered Harish-Chandra module with a real infinitesimal character to its (unique) lowest $K^{\sigma}$-type establish a bijection between equivalence classes of  irreducible tempered Harish-Chandra modules with a real infinitesimal character and $\widehat{K^{\sigma}}$.\label{th3}
\end{theorem}

It is useful to note that the inverse of the map in Theorem \ref{th3} is given by $$\mu \longmapsto J_{P_{\mu}}(\delta_{\mu},0)$$

\subsection{Classes of families of modules}

In this section for each $\mu\in \widehat{K^{\sigma}}$ we shall define a class of families  in  $\mathcal{M}(\g,\K)$.  These classes shall be used to construct  collection of correspondences between  $\widehat{G^{\sigma}}_{\mu}$ and $\widehat{K^{\sigma}\ltimes \fp^{\sigma}}_{\mu}$. We then conjecture that  each of  these correspondences induces an algebraic isomorphism between  $\widehat{G^{\sigma}}$ and $\widehat{K^{\sigma}\ltimes \fp^{\sigma}}$. The fundamental  idea is to build families by taking certain curves  of the form $\nu:\C\longrightarrow \fa^*$ lifting them to a family of Harish-Chandra modules for $(\g,\K)|_{\X_0}$ that is related  to  $\{I_{P_{\mu}}(\delta_{\mu},\nu(z))\}_{z\in \C}$ and then  compactify the family to obtain a family over $\X=\C\P^1$.

By Weyl's unitarity trick (e.g., \cite[Proposition 5.7]{Knapp86}) there is no difference between locally finite continuous representations of a compact group and algebraic representations of its complexification. From now on, whenever $K$ is a complex algebraic reductive group with a compact real form $K^{\sigma}$  we shall identify $\widehat{K^{\sigma}}$ with $\widehat{K}$, the set of equivalence classes of irreducible algebraic representations of $K$.  
Fix $\mu\in \widehat{K}$, by Vogan there is a corresponding  cuspidal parabolic subgroup of $G^{\sigma}$ that we denote by  $P_{\mu}$. Following tradition we denote its Langlands decomposition by   $M_{\mu}A_{\mu}N_{\mu}$ . To $P_{\mu}$ one can attach a $\Theta$-stable Cartan subalgebra $\fh_{\mu}$ of $\fg$. 

\begin{defn}
Fix $\mu\in \widehat{K}$. We define $\widetilde{\mathcal{M}}(\g,\K)_{\mu}$ to be the collection of all   $\mathcal{F}\in \mathcal{M}(\g,\K)$ that satisfy the following conditions.
\begin{enumerate}
\item $\F$ is generically irreducible. 
\item $\F$  is generated by $\F_{\mu}$.
\item $\mathcal{F}$ has an infinitesimal character with respect to $\h_{\mu}$.
\item  The collection of  composition factors of $\mathcal{F}|_{[1,0]}$  (with their multiplicities) contained in the collection of composition factors of   $I_{P_{\mu}}(\delta_{\mu},0)$ (with their multiplicities).
\end{enumerate}
\label{def521}\end{defn}
Note that $J_{P_{\mu}}(\delta_{\mu},0)$ can always be embedded in  $\mathcal{F}|_{[1,0]}$.
A family $\F\in \widetilde{\mathcal{M}}(\g,\K)_{\mu}$ should be thought of as a compactification of a subfamily  of $\{I_{P_{\mu}}(\delta_{\mu},\nu)\}_{\nu \in \fa^*}$.

\subsection{The correspondence using unique composition factor in a composition series}
If $\mathcal{F}\in \widetilde{\mathcal{M}}(\g,\K)_{\mu}$  the $K$-types of the fibers $\mathcal{F}|_{[\alpha:\beta]}$ are independent of  $[\alpha:\beta]\in \C\P^1$. Moreover, the multiplicity of $\mu$ is one and hence there is a unique irreducible composition factor in a composition series  of $\mathcal{F}|_{[\alpha:\beta]}$  that contains $\mu$. We shall denote it by $J_{\mu,[\alpha:\beta]}(\mathcal{F})$. \newline

\noindent \textbf{The correspondence:} Fix  a point $[\alpha:\beta]\in \R\P^1$ different from  $"0"=[1:0]$ and $"\infty"=[0:1]$. For any $\mu \in \widehat{K}$
we say that $\pi\in \widehat{G^{\sigma}}_{\mu}$ and $\eta \in \widehat{K^{\sigma}\ltimes \fp^{\sigma}}_{\mu}$  are in correspondence if there exists $\mathcal{F}\in \widetilde{\mathcal{M}}(\g,\K)_{\mu}$ with  $\pi \simeq J_{\mu,[\alpha:\beta]}(\mathcal{F})$ and $\eta \simeq J_{\mu,[1:0]}(\mathcal{F})$. 
\begin{conjecture}
As $\mu$ varies in $\widehat{K}$ and  $\mathcal{F}$ varies in $\widetilde{\mathcal{M}}(\g,\K)_{\mu}$ the correspondence   $J_{\mu,[\alpha:\beta]}(\mathcal{F}) \longleftrightarrow J_{\mu,[0:1]}(\mathcal{F})$ defines a bijection between $\widehat{G^{\sigma}}$ and $\widehat{K^{\sigma}\ltimes \fp^{\sigma}}$ such that:
\begin{enumerate}
\item The bijection is an algebraic isomorphism, namely: For each $\mu\in \widehat{K}$ the bijection restricts to an isomorphism of affine algebraic varieties between $\widehat{G^{\sigma}}_{\mu}$ and $ \widehat{K^{\sigma}\ltimes \fp^{\sigma}}_{\mu}$.
\item The bijection extends Vogan's bijection (Theorem  \ref{th3}) between the tempered dual of $\widehat{G^{\sigma}}$  with real infinitesimal character    to $\widehat{K^{\sigma}}\subset \widehat{K^{\sigma}\ltimes \fp^{\sigma}}$.
\item The bijection maps tempered representations to tempered representations.
\end{enumerate}
\label{conj}
\end{conjecture} 
 \begin{remark}
 It follows from  \cite{Vogan79,Vogan2007} that $\widehat{G^{\sigma}}_{\mu}$   can be  identified with the maximal ideals of a certain finitely generated abelian algebra of regular functions on $\fa_{\mu}^*$ and hence a complex affine algebraic variety. Similarly $ \widehat{K^{\sigma}\ltimes \fp^{\sigma}}_{\mu}$ is an affine algebraic variety. The structure of $\widehat{G^{\sigma}}_{\mu}$ and $ \widehat{K^{\sigma}\ltimes \fp^{\sigma}}_{\mu}$ as  affine algebraic varieties for complex semisimple  $G^{\sigma}$  is described in   \cite{Higson2011}.\label{r2}
  \end{remark}
  
Note that for any  $[\alpha:\beta]\in \R\P^1$ with $\alpha\beta\neq 0$ we have constructed a correspondence.  In the next section we shall see that for $SL_2(\R)$ any bijection between the admissible duals satisfying the properties listed in conjecture \ref{conj}  arises as one of  the constructed correspondences . Hence  for the case of $SL_2(\R)$ we characterize the constructed  collection of  bijections by the above mentioned three properties. This  suggest that the constructed correspondences could be characterized in the general case.

\subsection{The Jantzen filtration and the correspondence}\label{s5.4}
In this section we reformulate an additional   conjecture on how to calculate the various $J_{\mu,[\alpha:\beta]}(\mathcal{F})$ in term of  the Jantzen filtration. Recall that if $\sigma$ is a real structure of an algebraic family of Harish-Chandra pairs, $(\g,\K)$, then for any $\F\in \mathcal{M}(\g,\K)$ the $\sigma$-twisted dual $\F^{\sigma}$ is also in $\mathcal{M}(\g,\K)$  see \cite[sec. 2.4]{Ber2017}.
Let $\F\in \widetilde{\mathcal{M}}(\g,\K)_{\mu}$.  For any  nonzero  rational intertwining operator $\varphi$ from  $\F$  to $\F^{\sigma}$ and any $[\alpha:\beta]\in  \R\P^1$ there is a corresponding decreasing  filtration  $\{\F|_{[\alpha:\beta]}^n\}_{n\in \Z}$ of $F|_{[\alpha:\beta]}$ that is  independent of $\varphi$ (up to a shift of the filtration parameter) see \cite[sec. 4.1]{Ber2017}. 
We shall denote by $\widetilde{J}_{\mu,[\alpha:\beta]}(\mathcal{F})$ the unique Jantzen quotient of $F|_{[\alpha:\beta]}$ containing  $\mu$. For some $\F$ there are no nonzero rational intertwining operators from $\F$ into $\F^{\sigma}$.  For example in the case of $SL_2(\R)$ it exists if and only if   the function by which the Casimir section acts is real-valued on $\R\P^1$ (see \cite[prop. 3.5]{Ber2017}). The following  conjecture suggests a relatively easy method to calculate $ J_{\mu,[\alpha:\beta]}(\mathcal{F})$ in terms of the Jantzen quotients.

\begin{conjecture}
 For any $[\alpha:\beta]\in \R\P^1$ and any $\F\in \widetilde{\mathcal{M}}(\g,\K)_{\mu}$  for which $\F|_{[\alpha:\beta]}$ is reducible,     \hspace{3mm}  $\widetilde{J}_{\mu,[\alpha:\beta]}(\mathcal{F})\simeq J_{\mu,[\alpha:\beta]}(\mathcal{F})$.
\label{conj2}
\end{conjecture}

It should be stressed that in principle Jantzen quotients are much easier to calculate than composition factors. In addition to that,  Conjecture \ref{conj2}, if holds of course, says that the unique Jantzen quotient of  $\F|_{[\alpha:\beta]}$ that contains $\mu$  is irreducible.  
In the next section we shall see that both Conjecture \ref{conj} and Conjecture \ref{conj2} hold for $SL_2(\R)$.  We note that if $\widetilde{J}_{\mu,[\alpha:\beta]}(\mathcal{F})$ is irreducible  then $\widetilde{J}_{\mu,[\alpha:\beta]}(\mathcal{F})\simeq J_{\mu,[\alpha:\beta]}(\mathcal{F})$. Our strategy in the case of $SL_2(\R)$ will be to show that indeed $\widetilde{J}_{\mu,[\alpha:\beta]}(\mathcal{F})$ is irreducible.

\section{The $SL_2(\R)$-case}\label{sl2}

In this section after  describing the relevant families in the case of ${SL_2(\R)}$ we shall  calculate the various Jantzen quotients  $\widetilde{J}_{\mu,[\alpha:\beta]}(\mathcal{F})$, show that they coincide with ${J}_{\mu,[\alpha:\beta]}(\mathcal{F})$, and prove that  Conjectures \ref{conj} and \ref{conj2} hold for ${SL_2(\R)}$. 
\subsection{The family of Harish-Chandra pairs}\label{sec6.1}
Below we make explicit the deformation family of Harish-Chandra pairs for ${SL_2(\R)}$.
The described analysis  is   is parallel to the one in \cite[sec. 4.2-4.3]{Bernstein2016}). 
Through out this section  $G=SL_2(\C)$, the Cartan involution (of $SL_2(\R)$)  is given by $\Theta(A)=(A^t)^{-1}$ and its fixed point subgroup  is $K=SO(2,\C)$. A corresponding real form is given by   $\sigma(A)=\overline{A}$ and   its fixed point subgroup is $G^{\sigma}=SL_2(\R)$.
We shall fix bases $\{H\}$ of $\fk$ and $\{X,Y\}$ of $\fp$ satisfying 
\begin{eqnarray}
 && [H,X]=2X, [H,Y]=-2Y, [X,Y]=H.
 \label{com}
\end{eqnarray} 
For concreteness we  take 
\begin{eqnarray}\nonumber
&&H= \left(\begin{matrix}
0& -i\\
i& 0
\end{matrix}\right)
,\hspace{2mm} X= \frac{1}{2}\left(\begin{matrix}
1& i\\
-i& -1
\end{matrix}\right),\hspace{2mm} Y= \frac{1}{2}\left(\begin{matrix}
1& i\\
i& -1
\end{matrix}\right).
\end{eqnarray}
Each one  of the  vectors $H,X$, and $Y$ generates a subsheaf of $\g$ (in the sense of Section \ref{se3.1}) that is a  $K$-equivariant invertible sheaf. We denote these sheaves by $\g_{0}$, $\g_{2}$, and $\g_{-2}$, respectively. The action of   $\left(\begin{matrix}
\alpha & \beta\\
  -\beta  & \alpha 
\end{matrix}\right)\in SO(2,\C)$ on $\g_{n}$ is given by multiplication by $(\alpha +i \beta)^{n}$ for any $n\in \{-2,0,2\}$. The sheaf of Lie algebras corresponding to $\K$ coincides with $ \g_{0}$ and shall also be denoted    by $\k$. As $O_{\X}$-modules  $ \g_{0}=\k \simeq \O_{\X}$, $ \g_{\pm2}\simeq \O_{\X}(-1)$. 
The isotypic decomposition of $\g$ is given by  $\g=\g_{-2}\oplus \g_{0}\oplus \g_{2}$. 
More explicitly the section $H_0:=1\otimes H$ is a nowhere vanishing regular section of $\g_{0}$, the sections  $X_0:=1\otimes X$ and  $Y_0:=1\otimes Y$ are   nowhere vanishing rational sections of $\g_{2}$ and $\g_{-2}$, respectively. Moreover, the only pole of $X_0$  (and $Y_0$) is a simple pole at $\infty=[0:1]$.   The sections $H_0,X_0,Y_0$ form an $\mathfrak{sl}_2$-triplet and satisfy the same commutation relations as in 
(\ref{com}), but now these are commutation relations  between rational sections of $\g$.  
This completes the description of $\g$. The family $\K$ is the constant group scheme over $\X$ with fiber  $K=SO(2,\C)$.
\subsection{Generically irreducible families of Harish-Chandra modules} 
In this section we describe few invariants of generically irreducible families of Harish-Chandra modules that are generated by their minimal $K$-types and give their  classification for the deformation family of $SL_2(\R)$. The analysis is parallel to the case of the \textit{contraction family} of $SL_2(\R)$ that is given in \cite[sec. 4]{Bernstein2016}).
\subsubsection{The Casimir section}
Recall that the \textit{Casimir sheaf} $\mathcal{C}$ of $\g$ is the subsheaf of the center of the sheaf of enveloping algebras of $\g$ consisting of all sections of order 2 (with respect to the PBW filtration) that act trivially on the family of trivial modules  \cite[sec. 4.4.1]{Bernstein2016}). As a sheaf of $O_{\X}$-modules $\mathcal{C}\simeq O_{\X}(-2)$.   The canonical Casimir section  (a rational section) of $\mathcal{C}$ \cite[sec. 4.4.2]{Bernstein2016}) is given by 
\begin{eqnarray}
\Omega_0=H_0^2+2H_0+4Y_0X_0
\label{Cas}
\end{eqnarray}
This section satisfies  $\operatorname{ord}_{\infty}(\Omega_0)=-2$ and $\operatorname{ord}_{z}(\Omega_0)=0$ elsewhere on $\X$. 
\begin{lemma}\label{ca}
If $\F\in \mathcal{M}(\g,\K)$ is quasisimple then $\Omega_0$ acts on $\F$ by multiplication by a regular function of the form $c(r)=c_2r^2+c_1r+c_0\in \C[r]$.
\end{lemma}
\begin{proof}
Since $\Omega_0$ is a regular section over $\X_0$ then is must act by multiplication by some regular function $c(r)\in \C[r]$. Since $R^2\Omega_0$ is  a regular section over $\X_{\infty}$ then $R^2c(R^{-1})\in \C[R]$ and hence $c(r)=c_2r^2+c_1r+c_0$ for some $c_2,c_1,c_0\in \C$. 
\end{proof}
\begin{remark}
In general for the deformation family arising from a symmetric pair $(G,K)$ the Casimir section generates  a subsheaf that is isomorphic to $O_{\X}(-2)$ and  Lemma \ref{ca} holds.
\end{remark}

\subsubsection{The isotypic subsheaves}
As mentioned in Section \ref{sec4} each $\F\in \mathcal{M}(\g,\K)$  decomposes to a direct sum of $\K$-equivariant $\O_{\X}$-modules
   \[
   \mathcal F = \bigoplus _{n\in \widehat K}   \mathcal{F}_n .
   \]
 Here we identify the  algebraic dual of $K$ with $\Z$ and denote the parameter by $n$ instead of $\mu$.  
The sequence  $\{ \mathcal{F}_n\}_{n\in \Z}$ is an invariant of an equivalence class  of an algebraic family of Harish-Chandra modules for  $(\g,\K)$. If  $\mathcal{F}$ is generically irreducible then each non zero $\mathcal{F}_n$ is an invertible sheaf.

\subsubsection{Generically irreducible families generated by a minimal $K$-type}
For generically irreducible families generated by a fixed minimal $K$-type the Casimir action and the isotypic subsheaf corresponding to the minimal $K$-type completely determines the equivalence class.  Below we give a precise statement. For each  $m\in \Z$ we denote by  ${\mathcal{M}(\g,\K)}_m$ the set of equivalence classes of generically  irreducible families in $\mathcal{M}(\g,\K)$ that have $m$ as a minimal $K$-type and  are being  generated by their isotypic subsheaf corresponding to $m$. For $SL_2(\R)$ and its Cartan motion group an irreducible Harish-Chandra module is completely  determined by its $K$-types and the action of the center of the universal enveloping algebras. For generically irreducible families of Harish-Chandra modules the following analogous statement holds. 
\begin{theorem}
Let $m\in \Z$ and  $\F\in {\mathcal{M}(\g,\K)}_m$.  Then up to an isomorphism    $\F$ is completely determined by $\F_m$ and the action of the Casimir section. 
\end{theorem}
The proof  is analogous  to the one of  Theorem 4.9.3 of \cite{Bernstein2016}. The main change is  in  the action of the canonical Casimir section. In our case it must act  as multiplication by a  polynomial of degree bounded by two,  while in   \cite{Bernstein2016} it must act as multiplication by $c_1r+c_0+c_{-1}r^{-1}\in \C[r,r^{-1}]$. 
It should be noted that  $\F_m$ is isomorphic to $O_{\X}(l)\otimes_{\C}V_m$ for some $l\in \Z $ where $V_m$ is an irreducible representation of $K$ corresponding to $m$. Up to the action of the Picard group of $\X$    a family $\F\in {\mathcal{M}(\g,\K)}_m$ is determined by the action of the Casimir section. We list all possibilities in Table \ref{table1}.

\begin{table}[htbp]
\centering
\begin{tabular}{|c|c|c|c|}
\hline
   & \bf $K$-types &   {\bf Casimir}&  {\bf Parameters'}   \\ 
  &   &  $c(r)$ &  {\bf domain}   \\ \hline
\multirow{ 3}{*}{${ {\mathcal{M}(\g,\K)}_{0}}$} & $2\Z$ & $c_2r^2+c_1r+c_0$ & $c(r)\not \equiv k(k+2)$   for  \\
&&&  $0 \leq k\in \Z$  and even \\ \cline{2-4} 
   & $-k, -k{+}2,\dots , k$ &$k(k+2)$ & $0\leq k\in \Z$ and  even \\  \cline{3-4} \hline
\multirow{ 4}{*}{${ {\mathcal{M}(\g,\K)}_{1}}$ } & $2\Z+1$ & $c_2r^2+c_1r+c_0$ & $c(r)\not \equiv k(k+2)$   for    \\
&&& $-1\leq k\in \Z$   odd \\ \cline{2-4} 
   & $-k,-k{+}2,\dots , k$ &$k(k+2)$ & $0 \leq k\in \Z$  and  odd  \\  \cline{2-4} 
      & $1,3,5,\dots $ &$-1$ &  \\   \hline
\multirow{ 4}{*}{${ {\mathcal{M}(\g,\K)})_{-1}}$} & $2\Z+1$ & $c_2r^2+c_1r+c_0$ & $c(r)\not \equiv k(k+2)$   for    \\
&&& $-1\leq k\in \Z$   odd \\ \cline{2-4} 
   & $-k, -k{+}2,\dots , k$ &$k(k+2)$ & $0\leq k\in \Z$  and  odd  \\  \cline{2-4} 
      & $-1,-3,-5,\dots $ &$-1$ &  \\   \hline
      ${ {\mathcal{M}(\g,\K)}_{d}}$ & $d,d+2,d+4,\dots$ & $d(d-2)$ & $1<d \in \Z$ \\ \hline
           ${ {\mathcal{M}(\g,\K)}_{d}}$ & $d,d-2,d-4,\dots$ & $d(d+2)$ & $-1>d \in \Z$ \\ \hline
\end{tabular}
\caption{$K$-types and Casimir action for $\F\in {\mathcal{M}(\g,\K)}_m$}
\label{table1}
\end{table}

\subsubsection{Concerete bases}
In this section we show how the Casimir and $K$-types explicitly determine the action of $\g$ on $\F \in  {\mathcal{M}(\g,\K)}_{m}$ by means of choosing certain bases. 

Through out this section let $\F$ be in  ${\mathcal{M}(\g,\K)}_{m}$ with Casimir action that is given by  $c(r)$ and set of $K$-types   $I\subset \Z$. As we already mentioned $\F$ decomposes with respect to the action of $K$ as $\mathcal{F} = \bigoplus_{n\in I}\mathcal{F}_n$ and 
 each  of the nonzero isotypic subsheavs $\mathcal F_n$ is an invertible sheaf. For a nonzero  $\mathcal F_n$   we can choose  a rational section $f_n$  that is regular except perhaps at $r=\infty$, and also nowhere vanishing except perhaps    at $r=\infty$. Moreover   $\operatorname{ord}_{\infty} (f_n) = \deg (\mathcal F_n)$. The section $f_n$  is unique, up to a nonzero multiplicative factor.  For all of these facts see \cite[sec 4.6]{Bernstein2016}.  
The commutation relations of the   rational sections $X_0$, $H_0$ and $Y_0$  introduced in Paragraph~\ref{sec6.1} imply that we can rescale the sections $f_n$ such that 
\begin{eqnarray}\label{5}
&&H_0f_n  =nf_{n} \\
&&X_0f_n  =\begin{cases} 
 f_{n+2},& n+2\in I, m\leq n  \\
 \frac{1}{4}\left(c(r)-n(n+2) \right)f_{n+2},& n+2\in I, m>n 
\end{cases} \\
&&Y_0f_{n} =\begin{cases} 
  \frac{1}{4}\left(c(r)-n(n-2) \right)f_{n-2},& n-2\in I, m< n  \\
 f_{n-2},& n-2\in I, m\geq n \label{7}
\end{cases} 
\end{eqnarray}
hold, as long as $n\pm2 \in I$.

\subsection{The classes  $\widetilde{\mathcal{M}}(\g,\K)_{m}$ }
In this section we  explicitly describe $\widetilde{\mathcal{M}}(\g,\K)_{m}$. We shall start by proving that the generalized Harish-Chandra homomorphisms are regular. Then for each $m\in \Z=\widehat{K}$ we shall  specify the corresponding cuspidal parabolic $P_m=M_mA_mN_m$ with its $\Theta$-stable Cartan $\fh_m$, the corresponding discrete series representation $\delta_m\in \widehat{M_m}$, and the modules  $I_{P_{m}}(\delta_{m},0)$ and  $J_{P_{m}}(\delta_{m},0)$.  
\subsubsection{The generalized Harish-Chandra homomorphism} 
Here we show that the generalized Harish-Chandra homomorphism is a morphism of $O_{\X}$-modules.  

\begin{lemma}
Let  $\Omega$ be the Casimir of $\fg=\mathfrak{sl}_2(\C)$
\begin{enumerate}
\item $\Gamma(\X_{0},\mathcal{Z}(\g))=\C[r]\otimes_{\C}\C[\Omega]=\C[r,\Omega_0]$
\item  $\Gamma(\X_{\infty},\mathcal{Z}(\g))=\C[R,R^2\Omega_0]$
\end{enumerate}
\label{lem631}
\end{lemma}
\begin{proof}
By  Lemma (\ref{le}) $\mathcal{Z}(\g)\subset \mathcal{Z}(\g_{\text{const}})= O_{\X}\otimes_{\C}\mathcal{Z}(\fg)=O_{\X}\otimes_{\C}\C[\Omega]$, hence the first part of the lemma is clear. For the second part observe that $\Gamma(\X_{\infty},\mathcal{Z}(\g))$ is a submodule of the rank one free $\C[R]$-module, $\C[R]\otimes_{\C}\C[\Omega]$. Since $\C[R]$ is a principal ideal domain then  $\Gamma(\X_{\infty},\mathcal{Z}(\g))$ is also a rank one free $\C[R]$-module and hence have a basis that  consist of an element of the form $f(R)\otimes \Omega$ for some $f(R)\in \C[R]$.  From the realization of $\Omega$ as  $H^2+2H+4YX$ in terms of  bases of $\fk$ and $\fp$ that satisfy equations  \ref{com}  it follows that $f(R)$ must be of the form $cR^2$ for some $c\in \C^*$.  
\end{proof}

\begin{proposition}
For any $\Theta$-stable Cartan $\fh$ of  $\fg=\mathfrak{sl}_2(\C)$ the generalized Harish-Chandra homomorphism maps $\mathcal{Z}(\g)$ into $\mathcal{U}(\h)$.   
\end{proposition}
\begin{proof}
By Lemma \ref{lem631} it is enough to show that $\widetilde{\gamma}_{\fh}(\Gamma(\X_{\infty},\mathcal{Z}(\g))\subset \Gamma(\X_{\infty},\mathcal{U}(\h))$. Since $\Gamma(\X_{\infty},\mathcal{Z}(\g))=\C[R,R^2\Omega_0]$  it enough to show that $\widetilde{\gamma}_{\fh}(R^2\otimes \Omega)=R^2\otimes {\gamma}_{\fh}(\Omega)\in \Gamma(\X_{\infty},\mathcal{U}(\h))$. Since   $\Omega$ is in the space of elements of degree not more than two with respect to the usual PBW filtration then  $R^2\otimes {\gamma}_{\fh}(\Omega)\in \Gamma(\X_{\infty},\mathcal{U}(\h))$.
\end{proof}

\subsubsection{The case of $|m|>1$ (the discrete series case)}
If $m\in \Z$ with $|m|>1$ then by Vogan's description  of the admissible  dual we can take  $P_m=M_m=SL_2(\R)$. The subgroups $A_m$ and $N_m$ are equal to  the trivial group. The representation $\delta_m$ is the unique discrete series representation of $M_m=SL_2(\R)$  having $m$ as a minimal $K$-type. In this case $I_{P_{m}}(\delta_{m},0)=J_{P_{m}}(\delta_{m},0)=\delta_m$. The $\Theta$-stable Cartan $\fh_m$  is  $\mathfrak{so}(2,\C)=\fk=\operatorname{span}\{H\}$ and $\fn$ and $\overline{\fn}$ are given by  $\C X$ and $\C Y$,  respectively. 
Since  $\fh_m$ with $|m|>1$ arises from a compact Cartan of $\mathfrak{sl}_2(\R)$  we also denote it by $\fh_c$. The corresponding $\rho_c\in \fh_c^*$ is given by $\rho_{c}(H)=1$ and  the generalized Harish-Chandra homomorphism is determined by 
\begin{eqnarray}
\gamma_{c}(\Omega_0)=(H_0-1)^2+2(H_0-1)=H_0^2-1
\end{eqnarray}
In particular  $\gamma_{c}(R^2\Omega_0)=R^2\otimes (H_0^2-1)\in \Gamma(\X_{\infty},\mathcal{U}(\h_m))
$. Note that since $\fh_c=\fk$ then $\h_c\simeq O_{\X}$ so any morphism of sheaves of commutative $O_{\X}$-algebras    $\psi:\mathcal{U}(\h_c)\longrightarrow O_{\X}$ is given by $\psi(H_0)=\alpha$ for some $\alpha\in \C$.  

\begin{proposition}\label{p632}
Let $d\in \Z$ such that $|d|>1$ then $\widetilde{\mathcal{M}}(\g,\K)_{d}= { {\mathcal{M}(\g,\K)}_{d}}$.
\end{proposition}
\begin{proof} For any such $d$ up to the action of the Picard group of $\X$  there is exactly one equivalence class in ${\mathcal{M}}(\g,\K)_{d}$.  For simplicity we shall assume that $1<d$, the other case is proven similarly.
Let $\F \in {\mathcal{M}}(\g,\K)_{d}$ then  the $K$-types  of $\F$ are $d,d+2,...$ and  $\Omega_0$ acts by multiplication by $d(d-2)$.  Since $I_{P_{d}}(\delta_{d},0)$ is irreducible  (discrete series) it is isomorphic to   $\F|_0$. The only thing that is left to check  is that $\F$ has an infinitesimal character with respect to $\fh_c$. That is we ask whether there is $\alpha\in \C$  such that 
\begin{eqnarray}\nonumber
&&\psi(\gamma_{c}(\Omega_0))=\chi(\Omega_0) \Longleftrightarrow \alpha^2-1=d(d-2)
\end{eqnarray}
which obviously hold since we can take $\alpha=\pm (d-1)$ .
\end{proof}

\subsubsection{The case of $|m|\leq1$}
If $m\in \Z$ with $|m| \leq 1$ then by Vogan's description  of the admissible  dual we can take  $P_m$ to be  the subgroups of $SL_2(\R)$ consisting of upper triangular matrices. $M_m=\{\pm \mathbb{I}\}$, $A_m$ is the subgroup of diagonal matrices in $SL_2(\R)$ with positive entries and     $N_m$ the subgroup of $P_m$ with all diagonal entries equal to 1. The representation $\delta_0$ is the trivial representation and the representations $\delta_{\pm1}$ are the sign representation (the irreducible non trivial representation of the two elements group). $I_{P_{0}}(\delta_{0},0)$ is an irreducible spherical principal series representation (on which  $\Omega_0|_{[1:0]}$ acts by multiplication by $-1$) and hence $I_{P_{0}}(\delta_{0},0)=J_{P_{0}}(\delta_{0},0)$. The representations $I_{P_{1}}(\delta_{1},0)$ and $I_{P_{-1}}(\delta_{-1},0)$ coincide and are equal to   the nonspherical principal series representation that is the direct sum of the two limit discrete series. 
$J_{P_{1}}(\delta_{1},0)$ is the limit discrete series with minimal $K$-type $1$ and $J_{P_{-1}}(\delta_{-1},0)$ is the limit discrete series with minimal $K$-type $-1$.   The $\Theta$-stable Cartan $\fh_m$ is  the Lie algebra of diagonal matrices in $\mathfrak{sl}_2(\C)$ which is the  complexification of the Lie algebra of $A_m$. Since  $\fh_m$ with $|m|\leq1$ arises from a split Cartan of $\mathfrak{sl}_2(\R)$  we also denote it by $\fh_s$. Similarly, we shall use the notation  $P_s=M_sA_sN_s$ for $P_m=M_mA_mN_m$, with $|m|\leq 1$. The subalgebras $\fn$ and $\overline{\fn}$ are upper and lower triangular matrices (respectively)  with zeros along the diagonal. Explicitly 
\begin{eqnarray}\nonumber
&&\fh_s=\C{H}^{s},\hspace{2mm} \fn=\C X^{s},\hspace{2mm} \overline{\fn}=\C Y^{s}
\end{eqnarray}
where
\begin{eqnarray}\nonumber
&&{H}^{s}= \left(\begin{matrix}
1& 0\\
0& -1
\end{matrix}\right)
,\hspace{2mm} {X}^s= \left(\begin{matrix}
0& 1\\
0& 0
\end{matrix}\right),\hspace{2mm} {Y}^s= \left(\begin{matrix}
0& 0\\
1& 0
\end{matrix}\right)
\end{eqnarray}
The canonical Casimir section can be written as 
\begin{eqnarray}\nonumber
\Omega_0=1\otimes \left((H^s)^2+2H^s+4Y^sX^s\right)
\end{eqnarray}
The corresponding $\rho_s\in \fh_s^*$ is given by $\rho_{s}({H}^s)=1$ and  the generalized Harish-Chandra homomorphism is determined by 
\begin{eqnarray}\nonumber
\gamma_{s}(\Omega_0)=1\otimes (({H}^s)^2-1)
\end{eqnarray}
 In particular  $\gamma_{s}(R^2\Omega_0)=R^2\otimes (({H}^s)^2-1)\in \Gamma(\X_{\infty},\mathcal{U}(\h_m))
$. Note that since $\fh_s\subset \fp$ then $\h_s\simeq O_{\X}(-1)$ so any morphism of sheaves of commutative $O_{\X}$-algebras    $\psi:\mathcal{U}(\h_s)\longrightarrow O_{\X}$ is given by $\psi(H^s_0)=\alpha_1 r+\alpha_0$ for some $\alpha_0,\alpha_1 \in \C$ where as before  $H^s_0:=1\otimes H_0\in \Gamma(\X_0,\h_s)$. 

\begin{proposition}\label{p633}
$\widetilde{\mathcal{M}}(\g,\K)_{0}$ consists of all $\F$ in  $ { {\mathcal{M}(\g,\K)}_{0}}$ that have $2\Z$ as their set of $K$-types and  $c(r)=c_2r^2-1
$ with $c_2\in \C$.\label{pr632}
\end{proposition}
\begin{proof}
Let $\F\in \widetilde{\mathcal{M}}(\g,\K)_{0}$ then since $\F|_0\simeq I_{P_{0}}(\delta_{0},0)=J_{P_{0}}(\delta_{0},0)$ then  the set of $K$-types of $\F$  must be $2\Z$ and $c(r)$ must be of the form $c_2r^2+c_1r-1$. Since the family $\F$ has an infinitesimal character with respect to  $\h_s$ then $\psi(\gamma_{s}(\Omega_0))=\chi(\Omega_0) $ so there must be  $\alpha_0,\alpha_1\in \C$  such that 
\begin{eqnarray}\nonumber
&&(\alpha_1r+\alpha_0)^2-1=c_2r^2+c_1r-1\Longleftrightarrow \alpha_0=0, (\alpha_1)^2=c_2, c_1=0
\end{eqnarray}
so $c(r)=c_2r^2-1$.
\end{proof}

\begin{proposition}\label{p634}\noindent
\begin{enumerate}
\item $\widetilde{\mathcal{M}}(\g,\K)_{1}$ consists of all $\F$ in  $ { {\mathcal{M}(\g,\K)}_{1}}$  on which the Casimir acts by multiplication by  $c(r)=c_2r^2-1
$ with $c_2\in \C$ along with the following constrains on their $K$-types. If $c_2\neq 0$ the  set of $K$-types is 
$2\Z+1$ and if $c_2=0$ then the set of $K$-types  consists of all odd positive integers.
\item $\widetilde{\mathcal{M}}(\g,\K)_{-1}$ consists of all $\F$ in  $ { {\mathcal{M}(\g,\K)}_{-1}}$  on which the Casimir acts by multiplication by  $c(r)=c_2r^2-1
$ with $c_2\in \C$ along with the following constrains on their $K$-types. If $c_2\neq 0$ the  set of $K$-types is 
$2\Z+1$ and if $c_2=0$ then the set of $K$-types  consists of all odd negative integers.
\end{enumerate}
\end{proposition}
\begin{proof}
Let $\F\in \widetilde{\mathcal{M}}(\g,\K)_{1}$. Since  the collection of  composition factors of $\F|_0$  is contained in $\{ J_{P_{-1}}(\delta_{-1},0), J_{P_{1}}(\delta_{1},0)  \}$  then $c(0)=-1$ and the same calculation as in Proposition \ref{pr632} shows that $c(r)=c_2r^2-1$.   If $\F|_0\simeq J_{P_{1}}(\delta_{1},0)$ then as in the discrete series case $c_2$ must be zero. 
If $\F|_0\not\simeq J_{P_{1}}(\delta_{1},0)$ then  the set of $K$-types of $\F|_0$ (which is also the $K$-types of $\F$)  must be $2\Z+1$, since $I_{P_{1}}(\delta_{1},0)=J_{P_{-1}}(\delta_{-1},0)\oplus J_{P_{1}}(\delta_{1},0)$.  From the explicit description of the action of $\g$ (equations \ref{5}-\ref{7}) we see that generic irreducibility forces $c_2$ to be nonzero.
The second part of the proposition is proven similarly.
\end{proof}

\subsection{The admissible  duals}
In this section we briefly recall a classification of the admissible  dual of $SL_2(\R)$ and its Cartan motion group $SO(2)\ltimes \R^2$.

The sections $H_{\infty}=H_0=1\otimes H$, $X_{\infty}=RX_0:=R\otimes X$ and  $Y_{\infty}=RY_0:=R\otimes Y$ form a basis for  the Lie algebra $\Gamma(\X_{\infty},\g)$ over $\C[R]$. They satisfy the commutation relations 
\begin{eqnarray}
 && [H_{\infty},X_{\infty}]=2X_{\infty}, [H_{\infty},Y_{\infty}]=-2Y_{\infty}, [X_{\infty},Y_{\infty}]=R^2H_{\infty}
\end{eqnarray} 
The section 
\begin{eqnarray}\nonumber
\Omega_{\infty}:=R^2\Omega_0=R^2(H_{\infty}^2+2H_{\infty})+4Y_{\infty}X_{\infty}
\end{eqnarray} 
is a nonvanishing regular  section of the Casimir sheaf $\mathcal{C}$ over $\X_{\infty}$. In particular for every $R\in \C$ the image of  $\Omega_{\infty}$ in  the fiber $\mathcal{U}(\g|_R)$ is central and in fact it forms a basis for the free polynomial $\C$-algebra $\mathcal{Z}(\g|_R)$.  By essentially linear algebra, it can be shown that an irreducible admissible  $(\g|_R,K)$-module is completely determined by the action of $\Omega_{\infty}|_R$ (which is given by multiplication of a scalar) and a minimal $K$-type. Such calculations can be found for example in \cite[ch.1]{Vogan81}, \cite[sec. II. 1]{HoweTan} and \cite[ch. 8]{Taylor86}. 
 For  $R\neq 0$ it is more common to look on the action of the canonical Casimir $\Omega_0|_R=R^{-2}\Omega_{\infty}|_R$ instead of the action of $\Omega_{\infty}|_R$.  We denote by  $\mathcal{M}(\g|_R,K)$  the space of equivalence classes of irreducible admissible $(\g|_R,K)$-modules. For every $m\in \Z$ we denote by  $\mathcal{M}(\g|_R,K)_m$ the subspace of  $\mathcal{M}(\g|_R,K)$  consisting of all  modules having $m$ as a minimal $K$-type.  We introduce parameters for  $\mathcal{M}(\g|_R,K)_m$ as follows.
 For $R\neq0$ we shall denote by $(\omega,m)_R$ 
 the  (equivalence class of) irreducible admissible $(\g|_R,K)$-module on which $\Omega_0|_R=R^{-2}\Omega_{\infty}|_R$ acts by multiplication by $\omega $ and has minimal $K$-type $m$.
 For $R=0$ we shall denote by $(\omega,m)_0$ 
 the  (equivalence class of) irreducible admissible $(\g|_0,K)$-module on which $\Omega_{\infty}|_0$ acts by multiplication by $\omega $ and has minimal $K$-type $m$.
 We list all possibilities  for the case of   $R\neq 0$ corresponding to  $SL_2(\R)$  in Table \ref{table2} and  the case of   $R=0$ corresponding to  the Cartan motion group $SO(2)\ltimes \R^2$ in Table \ref{table3}.  Note that  the same $(\g|_R,K)$-module  can correspond to two different parameters. Two  different parameters   $(\alpha,m)_R$, $ (\beta,m')_R$  correspond to  isomorphic $(\g|_R,K)$-modules  if and only if $\{m,m'\}=\{1,-1\}$ and $\alpha=\beta$.  These are the only equivalences.

 Summarizing, for any  $R\in \C$ and   any $m\in \Z$ we can use the parameter $\omega$ in $(\omega,m)_R$ to  identify $\mathcal{M}(\g|_R,K)_m$ with a subset of $\C$. Explicitly, 
$$\mathcal{M}(\g|_R,K)_m\simeq \begin{cases} 
\C,& |m|>1\\
\{\text{point}\},& |m|\leq 1\end{cases}$$ 
Using this identification we can equip each $\mathcal{M}(\g|_R,K)_m$ with the structure of an affine algebraic variety (see remark \ref{r2} above). It should be stressed that $\mathcal{M}(\g|_R,K)_m$ has it own intrinsic structure as an affine algebraic variety and this structure coincides with the one we just obtained. We shall not prove these facts here.

\begin{table}[htbp]
\centering
\begin{tabular}{|c|c|c|c|}
\hline
   & \bf $K$-types &   {\bf Parameters}&  {\bf Parameters'}   \\ 
  &   &  $(R^{-2}\Omega_{\infty}|_R,m)_R$, $R\neq0$ &  {\bf domain}   \\ \hline
\multirow{ 3}{*}{$\mathcal{M}(\g|_R,K)_0$} & $2\Z$ & $(\omega,0)_R$ & $\omega\neq  k(k+2)$   for  \\
&&&  $0 \leq k\in \Z$  and even \\ \cline{2-4} 
   & $-k, -k{+}2,\dots , k$ &$( k(k+2),0)_R$ & $0\leq k\in \Z$ and  even \\  \cline{3-4} \hline
\multirow{ 4}{*}{$\mathcal{M}(\g|_R,K)_1$ } & $2\Z+1$ & $(\omega,1)_R$ & $\omega \neq k(k+2)$   for    \\
&&& $-1\leq k\in \Z$   odd \\ \cline{2-4} 
   & $-k, -k{+}2,\dots , k$ &$(k(k+2),1)_R$ & $0 \leq k\in \Z$  and  odd  \\  \cline{2-4} 
      & $1,3,5,\dots $ &$(-1,1)_R$ &  \\   \hline
\multirow{ 4}{*}{$\mathcal{M}(\g|_R,K)_{-1}$} & $2\Z+1$ & $(\omega,-1)_R$ & $\omega \neq  k(k+2)$   for    \\
&&& $-1\leq k\in \Z$   odd \\ \cline{2-4} 
   & $-k, -k{+}2,\dots , k$ &$(k(k+2),-1)_R$ & $0\leq k\in \Z$  and  odd  \\  \cline{2-4} 
      & $-1,-3,-5,\dots $ &$(-1,-1)_R$ &  \\   \hline
      $\mathcal{M}(\g|_R,K)_d$ & $d,d+2,d+4,\dots$ & $(d(d-2),d)_R$ & $1<d \in \Z$ \\ \hline
        $\mathcal{M}(\g|_R,K)_d$ & $d,d-2,d-4,\dots$ & $(d(d+2),d)_R$ & $-1>d \in \Z$ \\ \hline
\end{tabular}
\caption{The admissible dual of $SL_2(\R)$ by minimal $K$-types}
\label{table2}
\end{table}

\begin{table}[htbp]
\centering
\begin{tabular}{|c|c|c|c|}
\hline
   & \bf $K$-types &   {\bf Parameters}&  {\bf Parameters'}   \\ 
  &   &  $(\Omega_{\infty}|_0,m)_0$ &  {\bf domain}   \\ \hline
\multirow{ 2}{*}{$\mathcal{M}(\g|_0,K)_0$} & $2\Z$ & $(c,0)_0$ & $c\neq 0$     \\ \cline{2-4} 
   & $0$ &$(0,0)_0$&   \\  \cline{3-4} \hline
\multirow{ 2}{*}{$\mathcal{M}(\g|_0,K)_{1}$ } & $2\Z+1$ & $(c,1)_0$ & $c\neq 0$       \\  \cline{2-4} 
      & $1 $ &$(0,1)_0$ &  \\   \hline
\multirow{ 2}{*}{$\mathcal{M}(\g|_0,K)_{-1}$} & $2\Z+1$ & $(c,-1)_0$ & $c\neq 0$       \\ \cline{2-4} 
      & $-1 $ &$(-1,-1)_0$ &  \\   \hline
      $\mathcal{M}(\g|_0,K)_{d}$ & $d$ & $(0,d)_0$ & $1<d \in \Z$ \\ \hline
         $\mathcal{M}(\g|_0,K)_{d}$ & $d$ & $(0,d)_0$ & $-1>d \in \Z$ \\ \hline
\end{tabular}
\caption{The admissible dual of $ SO(2)\ltimes \R^2$ by minimal $K$-types}
\label{table3}
\end{table}

\subsection{The Jantzen quotients }
In this section following the calculations given in    \cite[sec. 5.2]{Ber2017}) we list all the Jantzen quotients for $\F$ in the various $\widetilde{\mathcal{M}}(\g,\K)_m$.
Let $\F\in \widetilde{\mathcal{M}}(\g,\K)_m$ on which  $\Omega_0$ acts  by multiplication by  $c(r)$. For any $K$-type $n$ of $\F$ we can find a non-vanishing   regular section $e_n$ such that $\Gamma(\X_{\infty},\F_n)=\C[R]e_n$ and the action of $\Gamma(\X_{\infty},\g)$  on $\Gamma(\X_{\infty},\F)$ is given by 
\begin{eqnarray}\nonumber
&&H_{\infty}e_n  =ne_{n} \\\nonumber
&&X_{\infty}e_n  =\begin{cases} 
 e_{n+2},& n+2\in I, m\leq n  \\
 \frac{R^2}{4}\left(c(R^{-1})-n(n+2) \right)e_{n+2},& n+2\in I, m>n 
\end{cases} \\\nonumber
&&Y_{\infty}e_{n} =\begin{cases} 
  \frac{R^2}{4}\left(c(R^{-1})-n(n-2) \right)e_{n-2},& n-2\in I, m< n  \\
 e_{n-2},& n-2\in I, m\geq n 
\end{cases} 
\end{eqnarray}
From these equations and the concrete expression for $c(r)$ which is either $c_2r^2-1$ or $m(m\pm2)$ (according to the value of $m$), we see that  $\F|_R$ is reducible for some $R\in \R$ if and only if $c(r)$ takes only real values on $\R$.   
From these formulas, doing the same calculations as in \cite[sec. 5.2]{Ber2017}), we obtain the following.
\begin{theorem}
For  $m\in \Z$,  $\F\in \widetilde{\mathcal{M}}(\g,\K)_{m}$  with Casimir action given by $c(r)=c_2r^2+c_0$, and any $[\alpha:\beta]\in \X_{\infty}$ we  have $$\widetilde{J}_{m,[\alpha:\beta]}(\mathcal{F})\simeq\begin{cases}
(\frac{c_2}{R^2}+c_0,m)_R & R\neq 0 \\
(c_2,m)_0 & R=0
\end{cases} $$
where $R=\frac{\alpha}{\beta}$.  \label{th4}
\end{theorem}
\noindent Note that in particular the Jantzen quotients are irreducible and as a result we obtain the following corollary. 
\begin{corollary}
Conjecture \ref{conj2} holds for $SL_2(\R)$.
\end{corollary}

\subsection{The Mackey-Higson-bijection for $SL_2(\R)$.} 

\begin{theorem}
Conjecture \ref{conj} holds for $SL_2(\R)$.
\end{theorem}

\begin{proof}
Since $\K$ is a constant family and since the Casimir of $SL_2(\R)$ is canonically defined regardless of any  concrete realization,  for any nonzero $R_1,R_2\in \X_{\infty}$ and any $m\in \Z$   there is a canonical identification   ${\mathcal{M}(\g|_{R_1},\K)}_{m}\simeq {\mathcal{M}(\g|_{R_2},\K)}_{m}$.
For each $0\neq R\in \X_{\infty}$ and  $m\in \Z$ the correspondence arising from Theorem \ref{th4} gives rise to a well-defined isomorphism of affine algebraic varieties  $\eta_m^R: { {\mathcal{M}(\g|_0,K)}_{m}}\longrightarrow  { {\mathcal{M}(\g|_R,\K)}_{m}}$, given by $\eta_m^R\left((z,m)_0\right)=(\frac{z}{R^2}+c_0,m)_R$.  We note that we can define a map  $\eta^R$  from  ${ {\mathcal{M}(\g|_0,K)}}$ into   ${ {\mathcal{M}(\g|_R,\K)}}$  by sending $(z,m)_0$ to $\eta_m^R\left((z,m)_0\right)=(\frac{z}{R^2}+c_0,m)_R$. Then  $\eta^R:=\Sigma_{m\in \Z} \eta_m^R$ is a well-defined bijection  from  ${ {\mathcal{M}(\g|_0,K)}}$ onto   ${ {\mathcal{M}(\g|_R,\K)}}$. It is well-defined since  $(z,m)_0\simeq  (z',m')_0$ if and only if  either $z=z'$ and $m=m'$ or $z=z'$ and $\{m,m'\}=\{-1,1\}$. In both cases  $\eta_m^R\left((z,m)_0\right )\simeq  \eta_{m'}^R\left((z',m')_0\right)$.  Since $$\eta_m^R\left((0,m)_0\right )=\begin{cases}
(-1,0)_R, & m=0\\
(-1,1)_R, & m=1\\
(-1,-1)_R, & m=-1\\
(m(m-2),m)_R, & m>1\\
(m(m+2),m)_R, & m<-1
\end{cases}
$$ 
it extends Vogan's bijection. Since the tempered dual of $(\g|_0,K)$ is given by $\{(0,m)_0|m\in \Z \}\cup \{ (c,\pm 1)_0| c<0\}\cup \{ (c,0)_0| c<0\}$  and the tempered dual of
$(\g|_R,K)$ is given by $\{(d(d-2),d)_R|0<d\in \Z \}\cup \{(d(d+2),d)_R|0>d\in \Z \} \cup \{ (\omega,\pm 1)_0| \omega < -1\}\cup \{ (\omega,0)_0| \omega \leq -1\}$
then $\eta^R$ takes  the tempered to tempered.  
\end{proof}
\begin{theorem}
Any bijection  between $\widehat{SL_2(\R)}^{\text{adms}}$ and $\widehat{SO(2)\ltimes \R^2}^{\text{adms}}$ satisfying the  three  conditions  in Conjecture \ref{conj}  arises form the correspondence  $J_{\mu,[\alpha:\beta]}(\mathcal{F}) \longleftrightarrow J_{\mu,[0:1]}(\mathcal{F})$ for some $[\alpha:\beta]\in \R\P^1$ with $\alpha \beta \neq 0$.
\end{theorem}
\begin{proof} 
Let $\Psi$ be  a bijection from  ${ {\mathcal{M}(\g|_0,K)}}$ onto   ${ {\mathcal{M}(\g|_1,\K)}}$ (the choice of  $R=1$ is not important) satisfying the  three  conditions in Conjecture  \ref{conj}. For any $d\in \Z$ with $|d|>1$,   ${ {\mathcal{M}(\g|_0,K)}}_d$ and   ${ {\mathcal{M}(\g|_1,\K)}}_d$ consist of exactly one equivalence class  and $\Psi$ must agree with any of the $\eta^R$ when restricted to  ${ {\mathcal{M}(\g|_0,K)}}_d$.  Now since $\Psi$ restricts to an algebraic isomorphism from   ${ {\mathcal{M}(\g|_0,\K)}}_0$ onto ${ {\mathcal{M}(\g|_1,\K)}}_0$  it must be of the form $\Psi((z,0)_0)=(\alpha z+\beta,0)_1$ for some $\alpha,\beta \in \C$ with $\alpha\neq 0$. Since $\Psi$ extends Vogan's bijection then  $\Psi((0,0)_0)=(-1,0)_1$ and hence  $\beta=-1$. Since $\Psi $ takes tempered modules to tempered modules it must map $\{(z,0)_0|z\in (-\infty,0]\}$ into $\{(z,0)_1|z\in (-\infty,-1]\}$ hence $\alpha >0$ and $\Psi$ coincides with $\eta^{\frac{1}{\sqrt{\alpha}}}$ on ${ {\mathcal{M}(\g|_0,\K)}}_0$. Similarly, it also coincide with it on ${ {\mathcal{M}(\g|_0,\K)}}_{\pm1}$.  
\end{proof}

\subsection{Comparison with the known bijection for the tempered dual}
In this section we compare our results in the case of  $SL_2(\R)$ to previous works. We explain how our tools can be used to shed new light on former  constructions of the Mackey bijection.

The bijection in the case of $SL_2(\R)$  for the tempered dual was considered before, see  \cite{George,Afgoustidis15,Qijun}, and in fact most of it was already given by Mackey  \cite{Mackey75}. Related calculations where done in \cite{Dooley83,Dooley85}. The bijection in these references agrees with the restriction of  our $\eta^{R=1}$  to the tempered dual. Our construction of the bijection is based on certain algebraic families of Harish-Chandra modules.

 In  \cite{Afgoustidis15} the bijection is defined with no reference to  families but  families are a central part of the story. More precisely   it is shown that for most $\pi_1\in\widehat{SL_2(\R)}^{\text{temp}}$ and $\pi_0\in \widehat{SO(2)\ltimes \R^2}^{\text{temp}}$ that correspond to each other under the Mackey bijection there is a family of representations $\{\pi_t\}_{t\in [0,\infty)}$ that extends $\pi_0$ and $\pi_1$ and,  in a certain sense, deform  the underlying Hilbert spaces of $\pi_0$ and $\pi_1$. This description is in agreement with \cite{Qijun} where the bijection is constructed via a family of D-modules on the flag variety of $SL_2(\R)$. We shall now write these families of representations in terms of parabolically induced representations and reinterpret  the formulas for the families in   \cite{Afgoustidis15} and \cite{Qijun}  using our language. 
\subsubsection{Unitary principal series}
Keeping the notations  of previous sections recall that $P_s$ is the minimal parabolic of upper triangular matrices with its Langlands decomposition $M_sA_sN_s$.
For $\epsilon \in \{0,1\}$ and $\lambda \in i\R$   we denote by $\operatorname{sgn}^{\epsilon}\otimes e^{\lambda}\otimes1$  the unitary irreducible representation of  $B$,  given by  
\begin{eqnarray}\nonumber
 &&\left(\begin{matrix}
-1& 0\\
0& -1
\end{matrix}\right)\longmapsto (-1)^\epsilon,\hspace{2mm}\left(\begin{matrix}
e^a& 0\\
0& e^{-a}
\end{matrix}\right)\longmapsto e^{\lambda a},\hspace{2mm}\left(\begin{matrix}
1& n\\
0& 1
\end{matrix}\right)\longmapsto 1
\end{eqnarray}
Let $I(\epsilon,\lambda)$ be  the  unitarily induced representation $\operatorname{Ind}_{P_s}^{SL_2(\R)}\operatorname{sgn}^{\epsilon}\otimes e^{\lambda}\otimes1$. The action of the Casimir on $I(\epsilon,\lambda)$ is given by $\lambda^2-1$ and the   infinitesimal character  (an element of $\fh_s$) is given by $H_s\longmapsto \lambda -1$.   Let $I_0(\epsilon,\lambda)$ be the corresponding representation of $SO(2)\ltimes \R^2$. 
Starting with $I(\epsilon,\lambda)$  with $\lambda\neq 0$  the recipe of  \cite{Afgoustidis15} (which is very similar to that of  \cite{Qijun}) is to look on the family $I_t:=I(\epsilon,\frac{\lambda}{t})$ with $t>0$ and then to show that  in a certain sense $I_t$ converges to  $I_0(\epsilon,\lambda)$ as $t$ goes to zero. This construction seems a bit  ad hoc (one can try to  take other functions of $t$) and  the infinitesimal character of $I_t$ that is given by $\frac{\lambda}{t}-1$  as well as the action of the Casimir that is given by $\frac{\lambda^2}{r^2}-1$, blow up as we approach $0$. These divergences already seen in the work of Dooley and Rice \cite{Dooley83,Dooley85} where they contract $I(\epsilon,\frac{\lambda}{t})$ to  $I_0(\epsilon,\lambda)$.  From our perspective  the same family of representation arises, as a family in one of the  $\widetilde{\mathcal{M}}(\g,\K)_{m}$. It  is essentially unique and the blow up in $t$ is merely an artifact of using non-regularized sections that are less suitable to describe the family. To see this, observe that the    minimal $K$-types of $I(\epsilon,\lambda)$  belong to  $\{-1,0,1\}$ and  the corresponding theta  stable Cartan is $\fh_s$. Following the scheme we presented above for $\eta^{R=1}$, we only need to consider  algebraic families of $(\g,\K)$-modules on which $\Omega_0$ acts via multiplication by $c(r)=c_2r^2-1$ and as a result    $\psi:\mathcal{U}(\h_s)\longrightarrow O_{\X}$ must be given   by $\psi(H^s_0)=\sqrt{c_2} r$ see Porpositions \ref{p633} and \ref{p634}. Demanding that the canonically defined Casimir element of $\mathfrak{sl}_2(\R)$ act via $\lambda^2-1$ at  $r=1$   we must take $c_2=\lambda^2<0$, $\psi(H^s_0)=\sqrt{c_2} r=\pm \lambda$. Note that the choice of sign for the square root is irrelevant since  $I(\epsilon,\lambda)\simeq I(\epsilon,-\lambda)$. In our coordinate system the Cartan motion group is obtained at  $r=\infty$ which is the same as $R=0$, in $\cite{Afgoustidis15,Qijun}$ it is obtained at $t=0$. Furthermore, from the way the Lie brackets scales with the parameter we must have $r=R^{-1}=ct^{-1}$ for some nonzero real constant $c$. We shall choose $c=1$ which corresponds to  $\eta^{R=1}$, the other choices lead to all other  Mackey-Higson bijections.    
  Even if one restricts his attention only to the tempered duals  there are two main advantages for our approach. 
\begin{enumerate}
\item The action of the Casimir and the infinitesimal character are completely determined  by demanding that the family belongs to one of the $\widetilde{\mathcal{M}}(\g,\K)_{m}$ (with appropriate $m\in \{0,\pm1\}$ according to the minimal $K$-type). Where each $\widetilde{\mathcal{M}}(\g,\K)_{m}$ is defined in terms of some characterizing properties. 
\item  In terms of the coordinate $R$ on $\X_{\infty}$ the action of the Casimir and the infinitesimal character are given by
\begin{eqnarray}\nonumber
&& \Omega_{\infty}=R^2\Omega_0\longmapsto R^2c(R^-1)=\lambda^2-R^2=c_2-R^2\\ \nonumber
&&  \psi(H^s_{\infty})=\psi(RH^s_{0})=\pm\lambda=\sqrt{c_2}
\end{eqnarray}
and there are no divergences.
\end{enumerate}
In the case of $I(0,0)$ following the same logic as above the Casimir  $\Omega_0$ must act via $c(r)=c_2r^2-1$  and also must satisfy $c(r=0)=-1$ hence $c_2=0$ and we are naturally led to a  family of $(\g,\K)$-modules that is  constant  over $\X_0$.  Similarly for the two limit discrete series that compose $I(1,0)$ we are forced to consider families that are constant over $\X_0$.   
  
\subsubsection{Discrete series}
In the case of a discrete series of $\mathfrak{sl}_2(\R)$, following the same logic as above and using Proposition \ref{p632}  again we are forced to consider families of $(\g,\K)$-modules that are  constant  over $\X_0$. All of the above is consistent with \cite{Afgoustidis15,Qijun,Dooley83,Dooley85} but from our perspective the behavior of the Casimir and the infinitesimal charactr it is forced on us by requiring the family to belongs to to one of the $\widetilde{\mathcal{M}}(\g,\K)_{m}$.

\section{Appendix: Generalized Harish-Chandra homomorphism for most compact Cartans}

Recall the we assume that $\fg$ is complex semisimple having a Cartan decomposition $\fg=\fk\oplus \fp$ and  $[\fk,\fk]\subset \fk$,   $[\fk,\fp]\subset \fp$, $[\fp,\fp]\subset \fk$. We also assume that $\fh$ is a $\Theta$-stable Cartan subalgebra of $\fg$ with $\fh=\ft\oplus \fa$, $\ft=\fk\cap \fh$, $\fa=\fp\cap \fh$. We  define an increasing filtration $\mathcal{U}(\fg)_{\fk}^n$  of $\mathcal{U}(\fg)$ as follows. 
\begin{eqnarray}
&&\mathcal{U}(\fg)_{\fk}^n=\mathcal{U}(\fg)_n \mathcal{U}(\fk)=\mathcal{U}(\fk)\mathcal{U}(\fg)_n\label{423}
\end{eqnarray}
where  $\mathcal{U}(\fg)_n$ is the standard  Poincare-Birkhoff-Witt filtration. Note that the second equality in \ref{423} follows from the commutation relations of $\fk$ and $\fp$.  
Obviously $\mathcal{U}(\fg)_{\fk}^n\subset \mathcal{U}(\fg)_{\fk}^{n+1}$,  $\mathcal{U}(\fg)_{\fk}^n\mathcal{U}(\fg)_{\fk}^m \subset \mathcal{U}(\fg)_{\fk}^{n+m}$, and $[\mathcal{U}(\fg)_{\fk}^n, \mathcal{U}(\fg)_{\fk}^m]\subset \mathcal{U}(\fg)_{\fk}^{n+m-1}$. 
  The PBW theorem implies that  the filtration is exhaustive, that is $\mathcal{U}(\fg)=\bigcup_{n\geq 0}\mathcal{U}(\fg)_{\fk}^n$. Moreover the associated graded algebra is canonically isomorphic to the enveloping algebra of  $\fk\ltimes \fg/\fk$ which is isomorphic to the fiber of $\mathcal{U}(\g)$ at $R=0$.
For any $n$ we denote the projection  $\mathcal{U}(\fg) \longrightarrow \mathcal{U}(\fg)/\mathcal{U}(\fg)_{\fk}^n$   by $\pi_n$. The order of any $0\neq \xi\in \mathcal{U}(\fg)$, denoted by $o(\xi)$, is defined to be the largest $n\in \mathbb{N}$ such that $\pi_n(\xi)\neq 0$.  Note that $o(\xi+\eta)\leq \max(o(\xi ),o( \eta))$.
Similarly,  we can define an increasing filtration of $\mathcal{U}(\fh)$ with respect to $\ft$. Since  $\fh$ is $\Theta$-stable and $\ft\subset \fk$  then $\mathcal{U}(\fg)^n_{\fk} \cap U(\fh)=\mathcal{U}(\fh)^n_{\ft}$. 
Any basis of $\fg$  that consists of eigenvectors of $\Theta$, induces via 
PBW theorem a basis for  each $\mathcal{U}(\fg)^n_{\fk} $. This is not true for an arbitrary basis of $\fg$. Explicitly,  let $d_{\fa}$, $d_{\ft}$, $d_{\fk}$ and $d_{\fp}$ be the dimensions of $\fa,\ft,\fk$ and $\fp$ respectively.   We shall pick an ordered basis $\{k_i\}_{i=1}^{d_{\fk}} $ for $\fk$ and an ordered basis $\{p_i\}_{i=1}^{d_{\fp}} $ for $\fp$ such that $\{k_i\}_{i=1}^{d_{\ft}} $ is a basis for $\ft$ and $\{p_i\}_{i=1}^{d_{\fa}} $ a basis for $\fa$. The PBW theorem and the commutation relations of $\fp$ and $\fk$ imply that  monomials of the form $p_1^{i_1}p_2^{i_2}...p_{d_{\fp}}^{i_{d_{\fp}}}k_1^{j_1}k_2^{j_2}...k_{d_{\fk}}^{j_{d_{\fk}}}$ with $|i|:=i_1+...+i_{d_{\fp}}\leq n$  and no restrictions on the $j$'s form a basis for the complex vector space $\mathcal{U}(\fg)_{\fk}^n$. Similarly, monomials of the form $p_1^{i_1}p_2^{i_2}...p_{d_{\fa}}^{i_{d_{\fa}}}k_1^{j_1}k_2^{j_2}...k_{d_{\ft}}^{j_{d_{\ft}}}$ with $|i|:=i_1+...+i_{d_{\fa}}\leq n$  and no restrictions on the $j$'s form a basis for the complex vector space $\mathcal{U}(\fh)_{\ft}^n$. 
Note that $o(p_1^{i_1}p_2^{i_2}...p_{d_{\fp}}^{i_{d_{\fp}}}k_1^{j_1}k_2^{j_2}...k_{d_{\fk}}^{j_{d_{\fk}}})=|i|$ and the order of a linear combination of such different monomials is the maximum of the order of the monomials with a nonzero coefficient. 
We observe that a non-zero section $f(R)\otimes \xi$ of $\Gamma(\X_{\infty},\mathcal{U}(\fg)_{\text{const}})$ belongs to   $\Gamma(\X_{\infty},\mathcal{U}(\g))$ if and only if $f(R)\in \C[R]$ is divisible  by $R^{o(\xi)}$. 
\begin{proposition}
Let $\fh$ be a $\Theta$-stable Cartan of $\fg$ with maximal dimension of $\fh\cap \fk$. Then  $\widetilde{\gamma}_{\fh}(\mathcal{Z}(\g))\subset \mathcal{U}(\h)$.
\end{proposition}

\begin{proof} Let $\fn$ and  $\overline{\fn}$ be  maximal nilpotent algebras as in section \ref{sec32}. By \cite[p.6]{Vogan79} we can assume that $\fn$ and  $\overline{\fn}$ are $\Theta$-stable.  Hence we can find ordered bases for $\fn$, $\overline{\fn}$, and $\fh$ which consists of vectors that are either in $\fk$ or in $\fp$.  We obtain an ordered  basis for $\fg$ by putting first the basis of  $\overline{\fn}$, then of $\fh$ and finally of $\fn$. By PBW theorem  we obtain a basis for $\mathcal{U}(\fg)$ that we shall denote by  $p_1^{i_1}p_2^{i_2}...p_{d_{\fp}}^{i_{d_{\fp}}}k_1^{j_1}k_2^{j_2}...k_{d_{\fk}}^{j_{d_{\fk}}}$ where all indeces take values in $\mathbb{N}_0$. Now it will be enough to show that for any $f(R)\otimes\xi \in \Gamma(\X_{\infty},\mathcal{Z}(\g))$ with  $\xi\in \mathcal{Z}(\fg)$ we have $\widetilde{\gamma_{\fh}}(f(R)\otimes\xi) \in \Gamma(\X_{\infty},\mathcal{U}(\h))$. Since $\widetilde{\gamma_{\fh}}(f(R)\otimes\xi)=f(R)\otimes\gamma_{\fh}(\xi) $ then it is enough to show that $o(\gamma_{\fh}(\xi))\leq o(\xi)$.   We expand $\xi$ in the above mentioned bases and write $\xi =\sum_{I,J} \alpha_{I,J}p_1^{i_1}p_2^{i_2}...p_{d_{\fp}}^{i_{d_{\fp}}}k_1^{j_1}k_2^{j_2}...k_{d_{\fk}}^{j_{d_{\fk}}}$  where we have used multi-index notation and $\alpha_{I,J}\in \C$.  Note that the affect of $\gamma_{\fh}$ on $\xi$ is to throw away all those monomials that do not lie inside $\mathcal{U}(\fh)$. Now $o(\xi)$  is the maximum value of $|I|$ that appear in the expression for $\xi$ in our bases and $o(\gamma_{\fh}(\xi))$ is the maximum value of $|I|$ that appear in the expression for $\gamma_{\fh}(\xi)$. Obviously  $o(\gamma_{\fh}(\xi))\leq o(\xi)$.
 \end{proof}

\bibliography{references}

\def\cprime{$'$}
\begin{thebibliography}{TYY17}

\bibitem[Afg15]{Afgoustidis15}
A.~Afgoustidis.
\newblock How tempered representations of a semisimple {L}ie group contract to
  its {C}artan motion group.
\newblock Preprint, 2015.
\newblock \href{https://arxiv.org/abs/1510.02650}{arXiv:1510.02650}.

\bibitem[AT16]{Adams2016}
J.~Adams and O.~Taibi.
\newblock Galois and cartan cohomology of real groups.
\newblock Preprint, 2016.
\newblock \href{https://arxiv.org/abs/1611.05956}{arXiv:1611.05956}.

\bibitem[BHS16]{Bernstein2016}
J.~Bernstein, N.~Higson, and E.~M. Subag.
\newblock Algebraic families of {H}arish-{C}handra pairs.
\newblock Preprint, 2016.
\newblock \href{https://arxiv.org/abs/1610.03435}{arXiv:1610.03435}.

\bibitem[BHS17]{Ber2017}
J.~Bernstein, N.~Higson, and E.~M. Subag.
\newblock Contractions of representations and algebraic families of
  harish-chandra modules.
\newblock Preprint, 2017.
\newblock \href{https://arxiv.org/abs/1703.04028}{arXiv:1703.04028}.

\bibitem[DR83]{Dooley83}
A.~H. Dooley and J.~W. Rice.
\newblock Contractions of rotation groups and their representations.
\newblock {\em Math. Proc. Cambridge Philos. Soc.}, 94(3):509--517, 1983.

\bibitem[DR85]{Dooley85}
A.~H. Dooley and J.~W. Rice.
\newblock On contractions of semisimple {L}ie groups.
\newblock {\em Trans. Amer. Math. Soc.}, 289(1):185--202, 1985.

\bibitem[Ful84]{Fulton84}
William Fulton.
\newblock {\em Introduction to intersection theory in algebraic geometry},
  volume~54 of {\em CBMS Regional Conference Series in Mathematics}.
\newblock Published for the Conference Board of the Mathematical Sciences,
  Washington, DC; by the American Mathematical Society, Providence, RI, 1984.

\bibitem[Geo09]{George}
Christopher~Yousuf George.
\newblock {\em The Mackey Analogy for $SL(n,\mathbb{R})$}.
\newblock PhD thesis, Penn State University, 2009.

\bibitem[Gil94]{Gilmore05}
Robert Gilmore.
\newblock {\em Lie groups, {L}ie algebras, and some of their applications}.
\newblock Robert E. Krieger Publishing Co., Inc., Malabar, FL, 1994.
\newblock Reprint of the 1974 original.

\bibitem[Hig08]{Higson08}
Nigel Higson.
\newblock The {M}ackey analogy and {$K$}-theory.
\newblock In {\em Group representations, ergodic theory, and mathematical
  physics: a tribute to {G}eorge {W}. {M}ackey}, volume 449 of {\em Contemp.
  Math.}, pages 149--172. Amer. Math. Soc., Providence, RI, 2008.

\bibitem[Hig11]{Higson2011}
Nigel Higson.
\newblock On the analogy between complex semisimple groups and their {C}artan
  motion groups.
\newblock In {\em Noncommutative geometry and global analysis}, volume 546 of
  {\em Contemp. Math.}, pages 137--170. Amer. Math. Soc., Providence, RI, 2011.

\bibitem[HT92]{HoweTan}
Roger Howe and Eng-Chye Tan.
\newblock {\em Nonabelian harmonic analysis}.
\newblock Universitext. Springer-Verlag, New York, 1992.
\newblock Applications of ${{\rm{S}}L}(2,{{\bf{R}}})$.

\bibitem[IW53]{Inonu-Wigner53}
E.~Inonu and E.~P. Wigner.
\newblock On the contraction of groups and their representations.
\newblock {\em Proc. Nat. Acad. Sci. U. S. A.}, 39:510--524, 1953.

\bibitem[Kna86]{Knapp86}
Anthony~W. Knapp.
\newblock {\em Representation theory of semisimple groups}, volume~36 of {\em
  Princeton Mathematical Series}.
\newblock Princeton University Press, Princeton, NJ, 1986.
\newblock An overview based on examples.

\bibitem[Kna02]{Knapp2002}
Anthony~W. Knapp.
\newblock {\em Lie groups beyond an introduction}, volume 140 of {\em Progress
  in Mathematics}.
\newblock Birkh\"auser Boston, Inc., Boston, MA, second edition, 2002.

\bibitem[KV95]{KnappVogan}
Anthony~W. Knapp and David~A. Vogan, Jr.
\newblock {\em Cohomological induction and unitary representations}, volume~45
  of {\em Princeton Mathematical Series}.
\newblock Princeton University Press, Princeton, NJ, 1995.

\bibitem[Mac75]{Mackey75}
George~W. Mackey.
\newblock On the analogy between semisimple {L}ie groups and certain related
  semi-direct product groups.
\newblock In {\em Lie groups and their representations ({P}roc. {S}ummer
  {S}chool, {B}olyai {J}\'anos {M}ath. {S}oc., {B}udapest, 1971)}, pages
  339--363. Halsted, New York, 1975.

\bibitem[OV90]{OnishchikVinberg}
A.~L. Onishchik and {\`E}.~B. Vinberg.
\newblock {\em Lie groups and algebraic groups}.
\newblock Springer Series in Soviet Mathematics. Springer-Verlag, Berlin, 1990.
\newblock Translated from the Russian and with a preface by D. A. Leites.

\bibitem[Qui]{Quillen}
D.~Quillen.
\newblock $``$1968 {Q}uillen notebooks$''$ on clay mathematics institute
  website.
\newblock
  \href{http://www.claymath.org/publications/quillen-notebooks}{http://www.claymath.org/publications/quillen-notebooks}.

\bibitem[Seg51]{Segal51}
I.~E. Segal.
\newblock A class of operator algebras which are determined by groups.
\newblock {\em Duke Math. J.}, 18:221--265, 1951.

\bibitem[Tay86]{Taylor86}
Michael~E. Taylor.
\newblock {\em Noncommutative harmonic analysis}, volume~22 of {\em
  Mathematical Surveys and Monographs}.
\newblock American Mathematical Society, Providence, RI, 1986.

\bibitem[TYY17]{Qijun}
Qijun Tan, Yi-Jun Yao, and Shilin Yu.
\newblock Mackey analogy via {$D$}-modules for {${\rm SL}(2,\Bbb{R})$}.
\newblock {\em Internat. J. Math.}, 28(7):1750055, 20, 2017.

\bibitem[Vog77]{Vogan77}
David~A. Vogan, Jr.
\newblock Classification of the irreducible representations of semisimple {L}ie
  groups.
\newblock {\em Proc. Nat. Acad. Sci. U.S.A.}, 74(7):2649--2650, 1977.

\bibitem[Vog79]{Vogan79}
David~A. Vogan, Jr.
\newblock The algebraic structure of the representation of semisimple {L}ie
  groups. {I}.
\newblock {\em Ann. of Math. (2)}, 109(1):1--60, 1979.

\bibitem[Vog81]{Vogan81}
David~A. Vogan, Jr.
\newblock {\em Representations of real reductive {L}ie groups}, volume~15 of
  {\em Progress in Mathematics}.
\newblock Birkh\"auser, Boston, Mass., 1981.

\bibitem[Vog00]{Vogan99}
David~A. Vogan, Jr.
\newblock The method of coadjoint orbits for real reductive groups.
\newblock In {\em Representation theory of {L}ie groups ({P}ark {C}ity, {UT},
  1998)}, volume~8 of {\em IAS/Park City Math. Ser.}, pages 179--238. Amer.
  Math. Soc., Providence, RI, 2000.

\bibitem[Vog07]{Vogan2007}
David~A. Vogan, Jr.
\newblock Branching to a maximal compact subgroup.
\newblock In {\em Harmonic analysis, group representations, automorphic forms
  and invariant theory}, volume~12 of {\em Lect. Notes Ser. Inst. Math. Sci.
  Natl. Univ. Singap.}, pages 321--401. World Sci. Publ., Hackensack, NJ, 2007.

\end{thebibliography}
\bibliographystyle{alpha}

\end{document}